\documentclass[11pt, reqno]{amsart}

\usepackage[section]{colepaperlibertine}

\title[The Bramson correction for nonlocal Fisher--KPP equations]{The Bramson correction for Fisher--KPP equations with nonlocal diffusion}

\author{Cole Graham}
\address{Department of Mathematics, Stanford University, 450 Jane Stanford Way, Building 380, Stanford, CA 94305, USA}
\email{\tt grahamca@stanford.edu}

\begin{document}

\begin{abstract}
  We establish the logarithmic Bramson correction to the position of solutions to the Fisher--KPP equation with nonlocal diffusion.
  Solutions with step-like initial data typically resemble a front at position ${c_* t - \frac 3 {2\lambda_*} \log t + \m O(1)}$ for explicit constants $c_*$ and $\lambda_*$.
  However, certain singular diffusions exhibit more exotic behavior.
\end{abstract}

\maketitle

\section{Introduction}

We study the Fisher--KPP equation on the line with nonlocal diffusion:
\begin{equation}
  \label{eq:main}
  \partial_t u = \mu \, (J \ast u - u) + f(u) \quad \textrm{for } (t, x) \in \R_+ \times \R.
\end{equation}
Here $\mu > 0$ is a constant, $J$ is a compactly-supported probability distribution on $\R$, and $f$ is a KPP reaction as defined below.
We supplement \eqref{eq:main} with the initial data $u(0, \anon) = \tbf{1}_{\R_-}$.
Then the solution $u$ will satisfy $0 \leq u \leq 1$.

If we replace the nonlocal diffusion $J \ast u - u$ by the Laplacian, we obtain the classical Fisher--KPP equation~\cite{Fisher, KPP}
\begin{equation}
  \label{eq:FKPP}
  \partial_t \upsilon = \mu \Delta \upsilon + f(\upsilon).
\end{equation}
This equation models numerous invasion phenomena, and the propagation of the solution is a subject of intense study.
We expect similar propagation in \eqref{eq:main}, and are thus interested in the ``position'' of $u$ as $t \to \infty$.
Precisely, define
\begin{equation*}
  \sigma_\theta(t) \coloneqq \sup\{x\in \R \mid u(t,x)\geq \theta\}
\end{equation*}
for $\theta \in (0,1)$.
Then $\sigma_\theta$ tracks the leading edge of $u$ at level $\theta$.
In this work, we study the long-time behavior of $\sigma_\theta$.

The nonlocal diffusion $J \ast u - u$ is the generator of a continuous-time random walk.
This suggests that \eqref{eq:main} admits a probabilistic representation.
And indeed, when the nonlinearity $f$ has a special form, \eqref{eq:main} is intimately related to a continuous-time \emph{branching} random walk.
This is a nonlocal version of the well-known relationship between the classical Fisher--KPP equation \eqref{eq:FKPP} and branching Brownian motion.
Bramson famously exploited this connection to determine the position of the classical solution $\upsilon$ to constant order~\cite{Bramson78, Bramson83}.
If $\mr \sigma_\theta$ denotes the leading edge of $\upsilon$, Bramson showed
\begin{equation}
  \label{eq:local-prop}
  \mr \sigma_\theta(t) = \mr c_* t - \frac{3}{2\mr \lambda_*} \log t + C_\theta + \smallO(1) \quad \textrm{as } t \to \infty
\end{equation}
for explicit constants $\mr c_*, \mr \lambda_* > 0$ which depend only on $\mu$ and $f'(0)$.

Of course, after a long time and suitable rescaling, random walks tend to resemble Brownian motion.
We therefore expect the nonlocal equation \eqref{eq:main} to exhibit similar behavior.
This matter has attracted a great deal of attention in \emph{discrete-time}~\cite{Biggins, Hammersley, Kingman, ABR1}.
This effort culminated in the masterful work of A\"{\i}d\'{e}kon~\cite{Aidekon}, who proved an analogue of \eqref{eq:local-prop} for discrete-time branching random walks.

In this work, we show
\begin{equation}
  \label{eq:nonlocal-prop}
  \sigma_\theta(t) = c_* t - \frac{3}{2 \lambda_*} \log t + \m O_\theta(1) \quad \textrm{as } t \to \infty
\end{equation}
for typical equations \eqref{eq:main}, where $c_* \in \R$ and $\lambda_* > 0$ depend on $\mu,$ $J$, and $f'(0)$.
We emphasize that \eqref{eq:main} only corresponds to a branching random walk when $f$ has a special form.
Precisely, in Section~\ref{sec:probability} we show that a probabilistic representation is equivalent to the following condition.
\begin{definition}
  \label{def:prob}
  A reaction $f$ is \emph{probabilistic} if it is analytic on the interval $(0, 1)$, ${f(0) = f(1) = 0}$, $f'(1^-) < 0$, and $(-1)^{k + 1} f^{(k)}(1^-) \geq 0$ for all $k \in \N_{\geq 2}$.
\end{definition}
\noindent
We demonstrate \eqref{eq:nonlocal-prop} for a much wider class of reactions, and thus extend this universal behavior beyond the probabilistic realm.

We now describe our hypotheses in detail.
We assume that the reaction $f$ is of KPP-type~\cite{KPP}:
\begin{enumerate}[label = \textup{(F\arabic*)}, leftmargin = 5em, labelsep = 1em, itemsep= 1ex, topsep = 1ex]
\item
  \label{hyp:reg}
  $f\in \m C^1([0,1])$ \enspace and \enspace $f\in \m C^{1,\gamma}$ near $0$ for some $\gamma\in \left(0, 1\right)$;
  
\item
  $f(0) = f(1) = 0$ \enspace and \enspace $f|_{(0,1)} > 0;$
  
\item
  \label{hyp:KPP}
  $f(u) \leq f'(0) u$ \enspace for all $u\in [0,1]$.
\end{enumerate}
In addition, our approach relies heavily on the compact support of $J$.
By rescaling space, we may assume:
\begin{enumerate}[label = \textup{(J1)}, leftmargin = 5em, labelsep = 1em, itemsep = 1ex, topsep = 1ex]
\item
  \label{hyp:compact}
  $\op{supp}J \subset [-1,1]$.
\end{enumerate}
We note that $u$ propagates quite differently  $J$ has sufficiently fat tails; see, for instance,~\cite{MK, Gantert, CR}.

Now let $\m P(\R)$ denote the set of Borel probability measures on $\R$ and $\delta_z \in \m P(\R)$ the unit point-mass at position $z \in \R$.
Suppose the measure $J$ has an atom at the origin of mass $a \in (0, 1)$, so that $J = (1 - a) \ti J + a \delta_0$ for some $\ti J \in \m P(\R)$ such that $\ti{J}(\{0\}) = 0$.
Then
\begin{equation*}
  \mu (J \ast u - u) = \mu [(1 - a) \ti J \ast u + a u - u] = (1 - a) \mu (\ti J \ast u - u).
\end{equation*}
Thus by rescaling $\mu$, we are free to assume that
\begin{enumerate}[label = \textup{(J2)}, leftmargin = 5em, labelsep = 1em, itemsep = 1ex, topsep = 1ex]
\item
  \label{hyp:no-zero-atom}
  $J(\{0\}) = 0.$
\end{enumerate}
That is, $J$ has no atom at $0$.
This convention is implicit in the classification of equations we present below.

We now turn to the propagation of $u$.
In analogy with the classical case, we expect $u$ to resemble a shift of a \emph{traveling front solution} to \eqref{eq:main}.
These solutions have the form $U_c(x - ct)$ for some speed $c \in \R$ and profile $U_c$ satisfying
\begin{equation*}
  \begin{gathered}
    \mu(J \ast U_c - U_c) + c \, U_c' + f(U_c) = 0,\\
    0 \leq U_c \leq 1, \quad U_c(-\infty) = 1, \quad U_c(+\infty) = 0.
  \end{gathered}
\end{equation*}
Under mild conditions on $J$, Coville, D\'{a}vila, and Mart\'{\i}nez~\cite{CDM} proved the existence of a minimal speed $c_* \in \R$ such that a monotone front $U_c$ exists for each speed $c \geq c_*$.
The minimal speed is given by
\begin{equation}
  \label{eq:FG}
  c_* = \inf_{\lambda > 0} \Gamma(\lambda) \quad \textrm{for} \quad \Gamma(\lambda) \coloneqq \frac 1 \lambda \left[\mu \int_{\R} \e^{\lambda x} \, J(\dn x) - \mu + f'(0)\right].
\end{equation}
We note that~\cite{CDM} considers fronts connecting $0$ to $1$ rather than $1$ to $0$, so their spatial signs are opposite ours.
Also,~\cite{CDM} assumes $J \in \m C^0$ and occasionally ${J \in \m C^1}$.
The differentiability of $J$ will be unnecessary for our purposes.
Furthermore, continuity can be removed by a limiting argument, as in~\cite{Coville}.

The speed formula \eqref{eq:FG} has a simple motivation.
The KPP condition \ref{hyp:KPP} means the front $U_c$ is \emph{pulled}: its behavior is determined by the leading edge $x \gg 1$, where $U_c \ll 1$.
In this regime, \eqref{eq:main} is well-approximated by its linearization about $0$.
Thus $c_*$ is the minimal speed of an exponential solution to the linearization of \eqref{eq:main}.
A brief calculation yields \eqref{eq:FG}.

We now define several classes of equations.
\begin{definition}
  \label{def:classification}
  A triple $(\mu, J, f)$ is \emph{regular} if $J(\R_+) > 0$ or $f'(0) < \mu$ and \emph{irregular} otherwise.
  An irregular triple is \emph{trapping} if ${f'(0) > \mu}$ and \emph{critical} if $f'(0) = \mu$.
\end{definition}
\noindent
We also apply these terms to the equation \eqref{eq:main} corresponding to a triple $(\mu, J, f)$.
This classification is motivated by the speed formula \eqref{eq:FG}.
We claim that an equation is regular if and only if the infimum in \eqref{eq:FG} is attained.

To see this, suppose $J$ has mass on $\R_+$.
Then the integral term in $\Gamma$ will grow exponentially as $\lambda \to \infty$, while $\Gamma(\lambda) \sim f'(0) \lambda^{-1}$ as $\lambda \to 0^+$.
It follows that the infimum is attained at some intermediate $\lambda_* \in \R_+$.
On the other hand, if $J(\R_+) = 0$, our normalization \ref{hyp:no-zero-atom} implies that the integral term in $\Gamma$ vanishes as $\lambda \to \infty.$
Thus if $f'(0) < \mu$, we have $\Gamma(\lambda) < 0$ for $\lambda$ sufficiently large.
Then the limits $\Gamma(0^+) = +\infty$ and $\Gamma(+\infty) = 0$ imply that $\Gamma$ attains its (negative) minimum.
In each case, the minimizer $\lambda_* \in \R_+$ is unique.
Indeed, $\lambda$ is a critical point of $\Gamma$ precisely when
\begin{equation*}
  \mu \int_{\R} \e^{\lambda x}(\lambda x - 1) \, J(\dn x) = f'(0) - \mu.
\end{equation*}
The left side is strictly increasing in $\lambda$, so $\Gamma$ has at most one critical point.
We establish \eqref{eq:nonlocal-prop} for regular equations.

Now suppose $(\mu, J, f)$ is irregular, so $J(\R_+) = 0$ and $f'(0) \geq \mu$.
Then $\Gamma > 0$ and $\Gamma(+\infty) = 0$.
Thus $c_* = 0$ and $J$ does \emph{not} attain its infimum.
Irregular equations fall naturally into two categories: trapping ($f'(0) > \mu$) or critical ($f'(0) = \mu$).
We show that trapping equations formally obey \eqref{eq:nonlocal-prop} with $c_* = 0$ and $\lambda_* = +\infty$.
That is, $\sigma_\theta(t) = \m O_\theta(1)$ for all $\theta \in (0, 1)$.
Thus solutions to \emph{not} propagate, but rather become trapped a bounded distance from the origin.
Critical equations, however, may exhibit more unusual dynamics.

\begin{remark}
  Discrete-time branching random walks obey a similar trichotomy.
  Most propagate linearly in time with a logarithmic correction, but some become trapped near the origin, and a few exhibit stranger behavior~\cite{ABR1, Bramson-loglog}.
  We discuss the correspondence between continuous- and discrete-time branching random walks in Section~\ref{sec:probability}.
\end{remark}

We now present our main results.
First, regular equations mimic the classical behavior \eqref{eq:local-prop}.
\begin{theorem}
  \label{thm:reg}
  Let $(\mu, J,f)$ be regular and let $\lambda_* \in \R_+$  denote the unique minimizer in \eqref{eq:FG}.
  Then for all $\theta \in (0,1)$, there exists a constant $C(\mu, J, f, \theta) > 0$ such that
  \begin{equation*}
    \Big| \sigma_\theta(t) - c_* t + \frac 3 {2\lambda_*} \log t \, \Big| \leq C \quad \textrm{for all } t \geq 2.
  \end{equation*}
\end{theorem}

Next, trapping equations resemble regular equations with $c_* = 0$ and ${\lambda_* = +\infty}$.
\begin{theorem}
  \label{thm:trapping}
  Let $(\mu, J, f)$ be trapping.
  Then for all $\theta \in (0, 1)$, there exists a constant $C(\mu, J, f, \theta) > 0$ such that
  \begin{equation*}
    \abs{\sigma_\theta(t)} \leq C \quad \textrm{for all } t \geq 0.
  \end{equation*}
\end{theorem}

\begin{remark}
  For probabilistic $f$, the conclusions of Theorems~\ref{thm:reg} and \ref{thm:trapping} follow from results of Addario-Berry and Reed~\cite{ABR1}.
  The principal contribution of the present work is the extension of these results to \emph{any} reaction satisfying \ref{hyp:reg}--\ref{hyp:KPP}.
  Our results are also similar to those of Gao~\cite{Gao}, who considered \eqref{eq:main} with an additional term of the form $\eps \Delta u$.
  Our methods, however, are quite different.
\end{remark}

We conclude with critical equations.
These are akin to the special discrete-time branching random walks considered in~\cite{Bramson-loglog}, and more exotic shifts are possible.
Rather than handling all critical equations, we detail a special case.
\begin{proposition}
  \label{prop:crit}
  Fix $\mu > 0$ and $p > 1$.
  Let $f(u) = \mu(u - u^p)$ and $J = \delta_{-1}$.
  Then for all $\theta \in (0, 1)$, there exists $C(\mu, p, \theta) > 0$ such that
  \begin{equation}
    \label{eq:crit}
    \Big| \sigma_\theta(t) + \frac{\log \log t}{\log p}\Big| \leq C \quad \textrm{for all } t \geq 2.
  \end{equation}
\end{proposition}
When $p = 2$, $f$ is the classical Fisher--KPP reaction, which corresponds to a binary branching random walk.
In this case, the conclusion of Proposition~\ref{prop:crit} follows from the main result of Bramson in~\cite{Bramson-loglog}.
However, the reaction $u - u^p$ is not probabilistic for any $p \neq 2$.
A walk with $p$ offspring at each branching event, for instance, has reaction $1 - u - (1 - u)^p$.
And indeed, by~\cite{Bramson-loglog}, the asymptotic position \eqref{eq:crit} is not attained by any branching random walk with jump-kernel $\delta_{-1}$.
Non-probabilistic equations thus exhibit a richer set of behavior.

Using the comparison principle, we can extend our results to solutions evolving from ``step-like'' initial data $u(0, \anon) = u_0$.
Indeed, if there exists $L \geq 0$ such that $0 \leq u_0 \leq 1$, $u_0|_{(-\infty, -L)} \equiv 1$, and $u_0|_{(L, \infty)} \equiv 0$, then we can sandwich $u$ between translations of the special step solution considered above.
Thus Theorem~\ref{thm:reg}, Theorem~\ref{thm:trapping}, and Proposition~\ref{prop:crit} all apply to $u$, with constants $C$ depending also on the initial data $u_0$.

Finally, we note that our model \eqref{eq:main} is distinct from the well-studied ``nonlocal Fisher--KPP equation,'' which involves a nonlocal \emph{nonlinearity} rather than nonlocal diffusion.
The nonlocal Fisher--KPP equation has garnered much attention in the last decade.
See \cite{BNPR} for traveling waves, \cite{BHR} for the Bramson correction to propagation, and \cite{ABBP} for a recent probabilistic interpretation.
There are two principal differences between \eqref{eq:main} and the nonlocal Fisher--KPP equation: \eqref{eq:main} obeys the comparison principle, but does not enjoy parabolic regularity.
The technical challenges in this work are thus quite different from those overcome in \cite{BHR}.

In Section~\ref{sec:probability}, we discuss the connection between \eqref{eq:main} and branching random walks.
We prove our main result, Theorem~\ref{thm:reg}, in Section~\ref{sec:regular}.
We handle irregular equations in Section~\ref{sec:irreg}.
In the appendix, we record a proof of a crucial but standard probabilistic estimate in our setting.

\section*{Acknowledgements}

This work was supported by the Fannie and John~Hertz Foundation and NSF grant DGE-1656518.
We warmly thank Lenya~Ryzhik for many productive discussions.

\section{The probabilistic connection}
\label{sec:probability}

In this section, we examine the relationship between the nonlocal Fisher--KPP equation \eqref{eq:main} and branching random walks (BRWs).

\subsection{Continuous time}

A continuous-time BRW is a growing collection of particles on $\R$, each jumping and reproducing independently with exponential rates $\mu$ and $r$, respectively.
When particles jump, we assume that they obey a law $J \in \m P(\R)$ satisfying \ref{hyp:compact} and \ref{hyp:no-zero-atom}.
When they reproduce, the particles have a random number of offspring distributed according to a law $\kappa \in \m P(\N_{\geq 2})$.
For a detailed description and construction of branching random walks, we refer the reader to Harris~\cite{Harris}.
Throughout, $\m X$ and $\m Z$ will denote random variables with laws $J$ and $\kappa$, respectively.
We will assume
\begin{equation}
  \label{eq:moment}
  \E \m Z^{1 + \gamma} < \infty
\end{equation}
for some $\gamma > 0$.
This condition is nearly sharp, as BRWs behave quite differently when $\E \m Z = \infty$; see, for instance,~\cite{Grey, RHRS}.

\begin{remark}
  More generally, we could allow a particle to have zero or one offspring when it branches.
  However, if a particle has one offspring, effectively nothing has changed.
  We can thus decrease $r$ and assume that $\P[\m Z = 1] = 0$.
  This is identical to our adjustment of $\mu$ to achieve \ref{hyp:no-zero-atom}.
  
  If $\m Z = 0$, a particle dies.
  This leads to the unpleasant possibility of extinction, in which all particles perish.
  However, conditional on non-extinction, such processes behave much like those without death.
  To avoid this technicality, we assume our particles are immortal.
  This explains the restriction $\kappa \in \m P(\N_{\geq 2})$.
\end{remark}

To understand the spreading of the population in a BRW, we study the particle with \emph{maximal} position.
That is, if $X_t^1, \ldots, X_t^{Z_t}$ denote the particle positions at time $t$, we study the cumulative distribution of the maximal particle:
\begin{equation*}
  v(t,x) \coloneqq \P\left[\max_{1\leq j \leq Z_t} X_t^j \leq x\right] \For (t,x)\in [0,\infty)\times\R.
\end{equation*}
Note that the population size $Z_t$ is itself random.

Let $g$ denote the probability generating function of $\kappa$:
\begin{equation*}
  g(s) \coloneqq \E s^{\m Z}.
\end{equation*}
A renewal argument along the lines of McKean~\cite{McKean} shows that $v$ solves a reaction-diffusion equation with nonlocal diffusion:
\begin{equation}
  \label{eq:main-reversed}
  \partial_t v = \mu[J \ast v - v] + r[g(v) - v].
\end{equation}

We find it more convenient to study the reversed cumulative distribution:
\begin{equation*}
  u(t, x) \coloneqq \P\left[\max_{1\leq j \leq Z_t} X_t^j > x\right] \For (t,x) \in [0,\infty) \times \R.
\end{equation*}
Then $u = 1 - v$, and \eqref{eq:main-reversed} suggests the definition
\begin{equation}
  \label{eq:reaction}
  f(u) \coloneqq r[1 - u - g(1 - u)].
\end{equation}
With this reaction, $u$ satisfies \eqref{eq:main}.
We will assume that the population begins from a single individual at the origin, so that $u(0, \anon) = \tbf{1}_{\R_-}$.

By \eqref{eq:moment}, $f$ satisfies \ref{hyp:reg}--\ref{hyp:KPP}.
In fact, $f$ satisfies the conditions in Definition~\ref{def:prob}.
Conversely, suppose a reaction $f$ is probabilistic in the sense of Definition~\ref{def:prob}.
If we define $r \coloneqq -f'(1^-) > 0$ and
\begin{equation*}
  g(s) \coloneqq s - \frac{f(1 - s)}{r},
\end{equation*}
then we can check that $g$ is analytic on $(0, 1)$, $g(0) = g'(0^+) = 0$, $g(1) = 1$, and $g^{(k)}(0^+) \geq 0$ for all $k \in \N_{\geq 2}$.
Hence $g$ is the probability generating function of some distribution $\kappa \in \m P(\N_{\geq 2})$.
We have thus shown that \eqref{eq:main} corresponds to a continuous-time BRW if and only if $f$ satisfies the conditions in Definition~\ref{def:prob}.

\subsection{Discrete time}

While continuous-time BRWs are of analytic interest, the majority of the BRW literature concerns \emph{discrete} time.
In this setting, each particle is replaced by an independent copy of a fixed point process $\Pi$ when we step forward in time.
For instance, $\Pi$ might be $\m Z$ particles independently sampled from $J$.

Crucially, we can obtain a discrete-time BRW from one in continuous time by sampling at evenly-spaced times.
In this case, the point process $\Pi$ is simply the set of particles in the continuous-time BRW after the first time interval.
The position of the maximal particle in discrete-time BRWs is well-understood; see, for instance,~\cite{Biggins, Hammersley, Kingman, BZ1, BZ2, ABR1, Aidekon}.
Thus when $f$ is probabilistic, our main results follow from prior work on discrete-time BRWs.

The literature on nonlocal reaction-diffusion equations is somewhat disconnected from this impressive body of probabilistic work.
It therefore seems desirable to explicitly establish the correspondence between results in discrete and continuous time.

We fix a continuous-time BRW with kernels $J$ and $\kappa$, and sample it at the discrete time-set $\Z_{\geq 0}$.
Let $\Pi$ denote the point process at time $t = 1$.
Each particle in $\Pi$ is individually distributed according to the law $J_1$ of a continuous-time random walk at time $1$.
Thus, $J_1$ is a Poissonization of $J$:
\begin{equation}
  \label{eq:Poisson}
  J_1 = \e^{-\mu} \sum_{k = 0}^\infty \frac{\mu^k}{k!}J^{\ast k}.
\end{equation}
The total number of particles in $\Pi$ is the population size $Z_1$ of the continuous-time BRW at time $1$.
Its law is not as easily described as $J_1$, but we can use a renewal argument to compute its moments.
In particular,
\begin{equation}
  \label{eq:growth}
  \E Z_t = \exp\left[r(\E \m Z - 1) t\right] \ForAll t \geq 0.
\end{equation}
We note that \eqref{eq:reaction} implies
\begin{equation}
  \label{eq:derivative}
  f'(0) = r(\E \m Z - 1).
\end{equation}
Thus $f'(0)$ represents the mean rate of particle production in the continuous-time BRW.

\begin{remark}
  The particles in $\Pi$ are correlated through their shared ancestries, so $\Pi$ is \emph{not} simply $Z_1$ particles independently sampled from $J_1$.
  Thus, this walk technically lies outside the scope of Addario-Berry and Reed~\cite{ABR1}.
  This issue is not serious, however, and we ignore it hereafter.
  The more general case is handled by A\"{\i}d\'ekon in~\cite{Aidekon}.
\end{remark}

With this setup, we compute the asymptotic speed of the maximal particle in the discrete-time BRW we have constructed.
This speed is related to the logarithmic moment generating function of $\Pi$:
\begin{equation*}
  R(\lambda) \coloneqq \log \E\left[\sum_{p \in \Pi} \e^{\lambda X(p)}\right],
\end{equation*}
where $p$ denotes a point in $\Pi$ with position $X(p)$.
Recalling that all the particles in $\Pi$ have law $J_1$, Wald's identity yields
\begin{equation*}
  \E\left[\sum_{p \in \Pi} \e^{\lambda X(p)}\right] = \E[Z_1] \, \E\big[\e^{\lambda X_1}\big],
\end{equation*}
where $\op{Law}(X_1) = J_1$.
Thus by \eqref{eq:growth},
\begin{equation*}
  R(\lambda) = \log \E Z_1 + \log \E \e^{\lambda X_1} = r(\E \m Z - 1) + \log \E \e^{\lambda X_1}.
\end{equation*}
Next, we use \eqref{eq:Poisson}:
\begin{equation*}
  \E \e^{\lambda X_1} = \e^{-\mu} \sum_{k = 0}^\infty \frac{1}{k!} \left(\mu \E \e^{\lambda \m X}\right)^k = \exp\left[\mu \left(\E \e^{\lambda \m X} - 1\right)\right].
\end{equation*}
Therefore
\begin{equation}
  \label{eq:LMGF-reduced}
  R(\lambda) = r(\E \m Z - 1) + \mu \left(\E \e^{\lambda \m X} - 1\right).
\end{equation}
As shown in~\cite{Shi}, for instance, the speed of the maximal particle in the discrete-time BRW is given by
\begin{equation}
  \label{eq:discrete-speed}
  c_* = \inf_{\lambda > 0} \frac{R(\lambda)}{\lambda} = \inf_{\lambda > 0} \frac{1}{\lambda} \left[r(\E \m Z - 1) + \mu \left(\E \e^{\lambda \m X} - 1\right)\right].
\end{equation}
Finally, \eqref{eq:derivative} shows that \eqref{eq:discrete-speed} agrees with \eqref{eq:FG}, which was derived by purely analytic means.

\subsection{Classification}

We close with a discussion of our classification of equations.
In discrete-time, there are also three fundamental classes of BRWs, at least when $J$ and $\kappa$ are sufficiently bounded.
In the regular case, \eqref{eq:discrete-speed} admits a minimizing $\lambda_*$, and Theorem~3 of~\cite{ABR1} states that the maximal particle has position
\begin{equation*}
  \sigma(n) = c_* n - \frac{3}{2\lambda_*} \log n + \m O(1).
\end{equation*}

In discrete-time, a BRW is regular if $X_1$ is unbounded from above.
After all, this ensures that $R$ grows superlinearly as $\lambda \to +\infty$, so that the infimum in \eqref{eq:discrete-speed} is attained.
We therefore consider the alternative: suppose $\sup X_1 < \infty$.
Since our $X_1$ is a Poissonization of $\m X$, this is equivalent to $\m X \leq 0$.
In fact, by \ref{hyp:no-zero-atom} it is equivalent to $\m X < 0$.
From here, the discrete-time classification hinges on the value of
\begin{equation*}
  \Xi \coloneqq \P[X_1 = 0] \, \E Z_1 = \exp\left[-\mu + r (\E \m Z - 1)\right] = \exp\big[f'(0) - \mu\big].
\end{equation*}
When $\Xi < 1$, Corollary~2 in~\cite{ABR1} implies that \eqref{eq:discrete-speed} has a minimizer.
In our case, this is equivalent to $f'(0) < \mu$.
We already showed that this implies the existence of a minimizer in \eqref{eq:FG}.

When $\Xi > 1$, Theorem~4 in~\cite{ABR1} states that the maximal particle remains a bounded distance from the origin.
That is, it becomes trapped.
In our case, of course, this is equivalent to the trapping condition $f'(0) > \mu$.

Finally, the borderline case $\Xi = 1$ is critical, and can yield unusual results.
Bramson neatly examined this situation in discrete time~\cite{Bramson-loglog}.
Due to the variety of possible behaviors, we do not comprehensively study the analogous $f'(0) = \mu$ case of \eqref{eq:main}.
Our Proposition~\ref{prop:crit} exhibits one family of critical shifts.

In summary, \eqref{eq:LMGF-reduced} allows us to translate between continuous and discrete times.
Our three main results then parallel a well-known trichotomy in discrete time, and follow from previous results when $f$ is probabilistic.
However, \eqref{eq:main} does not correspond to a branching process when $f$ lies outside this narrow class of reactions.
We therefore develop an alternative approach to the propagation of $u$.

\section{Regular equations}
\label{sec:regular}

The main virtue of the present work is the reduction of the regular problem to certain bounds on an ordinary \emph{non-branching} walk (namely, Lemma~\ref{lem:key} below).
This simplified approach is more flexible, and extends to non-probabilistic reactions.
We note that this reduction was previously observed by A\"{\i}d\'{e}kon and Shi~\cite{AS}.

\subsection{Proof outline}

To prove Theorem~\ref{thm:reg}, we follow the approach of Hamel, Nolen, Roquejoffre, and Ryzhik in~\cite{HNRR}.
There, the authors establish the Bramson shift for the classical Fisher--KPP equation using purely PDE tools.
They relate the solution in a moving frame to a linear Dirichlet problem on $\R_+$, and derive the shift from the long-time behavior of this linear problem.
We use the same method.

To begin, let $\sigma$ denote the expected position of the leading-edge of $u$:
\begin{equation}
  \label{eq:shift}
  \sigma(t) \coloneqq c_*t - \frac{3}{2\lambda_*} \log \frac{t + t_0}{t_0}
\end{equation}
for some regularizing time-shift $t_0 \geq 1$.
We analyze \eqref{eq:main} in a frame moving with $\sigma$.
We expect, although do not show, that $u$ eventually resembles the traveling front $U_{c_*}$ in this moving frame.
As shown in~\cite{CDM},
\begin{equation*}
  U_{c_*}(s) \asymp s \e^{-\lambda_* s} \quad \textrm{when } s \geq 1.
\end{equation*}
We thus broadly expect $u$ to decay like $\e^{-\lambda_* x}$ in the moving frame.
It is convenient to preemptively remove this decay.
Therefore, let
\begin{equation*}
  \bar v(t,x) \coloneqq \e^{\lambda_* x} u(t, x + \sigma(t)).
\end{equation*}
By standard manipulations, $\bar v$ satisfies
\begin{equation}
  \label{eq:transformed-upper}
  \partial_t \bar v = \nu K \ast \bar v + \dot \sigma (\partial_x \bar v - \lambda_* \bar v) + [f'(0) - \mu] \bar v + \e^{\lambda_* x} F\big(\e^{-\lambda_* x} \bar v\big),
\end{equation}
where
\begin{equation*}
  \nu \coloneqq \mu \int_{\R} \e^{\lambda_* x} \, J(\dn x), \quad K \coloneqq \frac{\mu}{\nu} \e^{\lambda_* x} J,
\end{equation*}
and $F(u) \coloneqq f(u) - f'(0) u \leq 0$ denotes the ``purely nonlinear'' part of $f$.
By construction, $K \in \m P(\R)$ is an exponential tilt of $J$.

By the definition of the shift $\sigma$,
\begin{equation*}
  \dot{\sigma}(t) = c_* - \frac{3}{2 \lambda_* (t + t_0)}.
\end{equation*}
If we discard the nonlinearity $F$ and terms of order $t^{-1}$ from \eqref{eq:transformed-upper}, we obtain its principal linear part
\begin{equation*}
  \m L \bar{v} \coloneqq \nu K \ast \bar{v} + c_* \partial_x \bar{v} + [f'(0) - \mu - c_* \lambda_*] \bar{v}.
\end{equation*}
By the definition of $c_*$ and $\lambda_*$,
\begin{equation*}
 \nu = c_* \lambda_* + \mu - f'(0) \And \nu \int_{\R} x \, K(\dn x) = c_*.
\end{equation*}
Let
\begin{equation*}
  m \coloneqq \frac{c_*}{\nu}
\end{equation*}
denote the mean of the probability distribution $K$.
Then we can write
\begin{equation*}
  \m L \bar{v} = \nu\left[K \ast \bar{v} + m \, \partial_x \bar{v} - \bar{v}\right].
\end{equation*}
It follows that $\m L 1 = \m L x = 0$ and $\m L x^2 = \nu \op{Var} K > 0$.
Thus $\m L$ resembles a multiple of the Laplacian to second order, and the principal part of \eqref{eq:transformed-upper} is a nonlocal analogue of the heat equation.

The remaining linear part in \eqref{eq:transformed-upper} is due to the logarithmic term in $\sigma$:
\begin{equation*}
  \frac{3}{2(t + t_0)} \bar v - \frac{3}{2\lambda_*(t + t_0)} \partial_x \bar v.
\end{equation*}
The first term corresponds to multiplication by the factor $(t + t_0)^{\frac 3 2}$.
That is, it could be trivially removed by replacing $\bar{v}$ by $(t + t_0)^{-\frac{3}{2}} \bar{v}$.
The second term \emph{should} be negligible, but is technically more difficult to handle.
We therefore study the Dirichlet problem
\begin{equation}
  \label{eq:Dirichlet}
  \begin{cases}
    \partial_t z = \m L z + \frac{D}{t + 1} \partial_x z & \textrm{on } \R_+,\\
    z = 0 & \textrm{on } (-\infty, 0],\\
    z(0, x) = \tbf{1}_{(L, 2L)}(x),
  \end{cases}
\end{equation}
for some fixed $D \in \R$ and $L \gg 1$ to be determined.
Note that we have replaced the time-shift $t_0$ by $1$ in this problem.
We will use the degree of freedom afforded by $t_0$ to relate a time-shift of \eqref{eq:Dirichlet} to \eqref{eq:transformed-upper}.

The Dirichlet model \eqref{eq:Dirichlet} is closely related to a random walk with killing.
Let $\bar{K}$ denote the spatial reverse of the measure $K$, so that $\bar{K}(A) = K(-A)$ for every Borel $A \subset \R$.
Then let $(X_s)_{s \geq 0}$ perform a continuous-time random walk with jump rate $\nu$, jump law $\bar{K}$, and constant drift $c_*$ starting from $0$.
By the construction of $K$, $\E X_s = 0$ for all $s \geq 0$.
The process $(X_s)_{s \geq 0}$ is simply a centered walk with jump law $K$ viewed \emph{backwards} in time.
We use it in a Feynman--Kac representation of \eqref{eq:Dirichlet}.

To do so, we must account for the extra drift $\frac{D}{t + 1} \partial_x$.
The time-dependence of this drift somewhat complicates matters.
Let us fix $(t, x) \in [0, \infty) \times \R_+$, and define the log-drifting walk
\begin{equation*}
  Y_s^x \coloneqq X_s + x + D \log \frac{t + 1}{t - s + 1} \quad \textrm{for } 0 \leq s \leq t.
\end{equation*}
Then \eqref{eq:Dirichlet} admits a Feynman--Kac representation via $Y$:
\begin{equation}
  \label{eq:FK}
  z(t, x) = \P\Big[Y_t^x \in (L, 2L),\; Y_s^x > 0 \; \textrm{ for all } \; 0 \leq s \leq t\Big].
\end{equation}
To construct super- and subsolutions for \eqref{eq:transformed-upper}, we use the behavior of $Y$ to control $z$.
The following lemma is the key to our comparison arguments.
\begin{lemma}
  \label{lem:key}
  There exists an initial length $L > 0$ and a constant $C_*(\mu, J, f'(0), D) \geq 1$ such that for all $(t, x) \in \R_+ \times \R_+$,
  \begin{equation}
    \label{eq:key}
    \frac{(x - C_*)}{C_*(t + 1)^{\frac 3 2}} \tbf{1}_{x \leq \sqrt{t}} \leq z(t, x) \leq \frac{C_*(x + 1)}{(t + 1)^{\frac 3 2}}.
  \end{equation}
\end{lemma}
We recall that $L$ is the width of the initial data $\tbf{1}_{(L, 2L)}$ in \eqref{eq:Dirichlet}.
For the remainder of the paper, we let it assume the value given by Lemma~\ref{lem:key}.

\begin{remark}
  This decay at rate $t^{-\frac 3 2}$ cancels the time-dependent growth term $\frac{3}{2(t + t_0)} \bar v$ in \eqref{eq:transformed-upper}, which is due to $- \lambda_* \dot \sigma \bar v$.
  This ensures that $u$ is order $1$ at position $\sigma$.
  Thus the time-decay in Lemma~\ref{lem:key} justifies the coefficient $-\frac{3}{2 \lambda_*}$ of the Bramson correction.
\end{remark}

By \eqref{eq:FK}, Lemma~\ref{lem:key} belongs to a family of ``ballot theorems'' widely used in the theory of branching processes; see, for instance, the survey~\cite{ABR2}.
In particular, our lemma is simply a continuous-time version of Lemma~3.2 in~\cite{Mallein}.
For the sake of completeness, we prove Lemma~\ref{lem:key} in the appendix.

\subsection{An upper bound}

With Lemma~\ref{lem:key}, we can easily construct a supersolution for $\bar v$.
Let $D = -\frac{3}{2\lambda_*}$ in \eqref{eq:Dirichlet}, and define $I \coloneqq [C_* + 1, C_* + 2]$ with the constant $C_*$ given by Lemma~\ref{lem:key}.
Then the lower bound in Lemma~\ref{lem:key} implies the existence of $\delta, T > 0$ such that
\begin{equation}
  \label{eq:lower-interval}
  (t + 1)^{\frac 3 2} z(t, x) \geq \delta
\end{equation}
for all $(t, x) \in [T, \infty) \times I$.
We define
\begin{equation*}
  \bar w(t, x) \coloneqq \delta^{-1} (t + T + 1)^{\frac 3 2} z(t + T, x + C_* + 2)
\end{equation*}
so that \eqref{eq:lower-interval} becomes
\begin{equation*}
  \bar w \geq 1 \quad \textrm{on } [0, \infty) \times [-1, 0].
\end{equation*}
By the upper bound in Lemma~\ref{lem:key}, there exists $C > 0$ such that
\begin{equation}
  \label{eq:super-upper}
  \bar w(t, x) \leq C (x + 1) \ForAll (t, x) \in [0, \infty) \times [0, \infty).
\end{equation}

Now, $\bar w$ nearly solves \eqref{eq:transformed-upper} on $[0, \infty)$.
It is only missing the negative nonlinearity $F$, so $\bar w$ is a supersolution to \eqref{eq:shift}, provided we take $t_0 = T + 1$.
Also,
\begin{equation*}
  \bar{w} \geq 1 \geq \bar v
\end{equation*}
on the augmented boundary $[0, \infty) \times [-1, 0]$.
This is the region that the nonlocal kernel $K$ ``sees'' from within $\R_+$.
Furthermore,
\begin{equation*}
  \bar w(0, \anon) \geq 0 \geq \bar v(0, \anon) \quad \textrm{on } [0, \infty).
\end{equation*}
Therefore, the comparison principle implies
\begin{equation}
  \label{eq:super-comp}
  \bar w \geq \bar v \quad \textrm{on } [0, \infty) \times [0, \infty).
\end{equation}

Returning to our solution $u$, \eqref{eq:super-upper} and \eqref{eq:super-comp} yield
\begin{equation*}
  u(t, x + \sigma(t)) \leq \e^{-\lambda_* x} \bar w(t, x) \leq C (x + 1) \e^{-\lambda_* x}
\end{equation*}
for all $(t, x) \in [0, \infty) \times [0, \infty)$.
Since the right side vanishes in the $x \to \infty$ limit,
\begin{equation*}
  \sigma_\theta(t) \leq \sigma(t) + C_\theta
\end{equation*}
for all $\theta \in (0, 1)$ and some $C_\theta > 0$.

\subsection{A lower bound}

We now construct a subsolution to \eqref{eq:main} to establish the lower bound in Theorem~\ref{thm:reg}.
For the upper bound, we studied $\e^{\lambda_* x} u$ in the moving frame $c_* t - \frac{3}{2\lambda_*}\log\frac{t + t_0}{t_0}$.
This was chosen so that solutions to a corresponding linear Dirichlet problem remain bounded in time away from $0$ and $\infty$ (locally in space).

We consider a similar transformation in this section, but must now contend with the nonlinear absorption.
To make the nonlinearity negligible, we'd like $u$ to be small.
Following~\cite{BHR}, we use a different logarithmic shift, to induce polynomial decay in time.
Fix
\begin{equation*}
  D_\gamma > \max\left\{\frac{1}{\lambda_*} \left(\frac{1}{\gamma} - \frac{3}{2}\right), 0\right\},
\end{equation*}
where $\gamma \in (0, 1)$ is the H\"{o}lder exponent from \ref{hyp:reg}.
Then we study
\begin{equation*}
  \ubar v(t, x) \coloneqq \e^{\lambda_* x} u\left(t, x + c_*t + D_\gamma \log(t + 1)\right),
\end{equation*}
which satisfies
\begin{equation}
  \label{eq:transformed-lower}
  \partial_t \ubar v = \m L \ubar v + \frac{D_\gamma}{t + 1} (\partial_x \ubar v - \lambda_* \ubar v\,) + \e^{\lambda_* x} F\big(\e^{-\lambda_* x} \ubar v\,\big).
\end{equation}

Now let $z$ solve \eqref{eq:Dirichlet} with $D = D_\gamma$.
By Lemma~\ref{lem:key},
\begin{equation*}
  z(t, x) \leq \frac{C_* (x + 1)}{(t + 1)^{\frac{3}{2}}} \quad \textrm{on } [0, \infty) \times [0, \infty).
\end{equation*}
Thus
\begin{equation*}
  \zeta(t, x) \coloneqq (t + 1)^{-\lambda_* D_\gamma} z(t, x)
\end{equation*}
solves the linearization of \eqref{eq:transformed-lower} and satisfies
\begin{equation}
  \label{eq:poly-decay}
  \zeta(t, x) \leq \frac{C_* (x + 1)}{(t + 1)^\beta}
\end{equation}
for
\begin{equation*}
  \beta \coloneqq \frac{3}{2} + \lambda_* D_\gamma > \frac{1}{\gamma}.
\end{equation*}

We cannot simply use $\zeta$ as a subsolution, since the nonlinearity $F$ in \eqref{eq:transformed-lower} is negative.
Therefore define
\begin{equation*}
  \ubar w_+(t, x) \coloneqq a(t) \zeta(t, x + 2L)
\end{equation*}
for some decreasing temporal profile $a$ to be determined.
For $\ubar w_+$ to be a subsolution to \eqref{eq:transformed-lower}, we require
\begin{equation*}
  \frac{\dot a}{a} \ubar w_+ \leq \e^{\lambda_* x} F(\e^{-\lambda_* x} \ubar w_+).
\end{equation*}
By \ref{hyp:reg}, there exists $C_F > 0$ such that
\begin{equation*}
  \abs{F(s)} \leq C_F s^{1 + \gamma}.
\end{equation*}
Recall that $\zeta(t, x) = 0$ for all $x \leq 0$.
Thus by \eqref{eq:poly-decay},
\begin{align*}
  \ubar w_+^{-1} \e^{\lambda_* x} \big|F\big(\e^{-\lambda_* x} \ubar w_+\big)\big| &\leq C_F \e^{-\gamma \lambda_* x} \ubar w_+^\gamma\\
                                                                 &\leq C_F C_*^\gamma \e^{-\gamma \lambda_* x} (x + 2L + 1)^\gamma a(t)^\gamma (t + 1)^{-\beta \gamma}\\
                                                                 &\leq C a(t)^\gamma (t + 1)^{-\beta \gamma}
\end{align*}
for some $C > 0$.
It thus suffices to let $a$ solve
\begin{equation*}
  \dot a = -C a^{1 + \gamma} (t + 1)^{-\beta \gamma}.
\end{equation*}
Because $\beta \gamma > 1$, positive solutions will remain uniformly bounded away from $0$.
We choose the solution with $a(0) = \e^{-\lambda_* L}$, so that
\begin{equation*}
  \ubar w_+(0, \anon) \leq \ubar v(0, \anon).
\end{equation*}
Then by the comparison principle,
\begin{equation*}
  \ubar w_+ \leq \ubar v.
\end{equation*}

We now use the lower bound in Lemma~\ref{lem:key}.
By \eqref{eq:key} and our construction of $\ubar w_+$, there exist $\delta, T > 0$ such that
\begin{equation*}
  \ubar v\big(t, \sqrt{t} - D_\gamma \log (t + 1)\big) \geq \delta t^{\frac{1}{2} - \beta} \ForAll t \geq T.
\end{equation*}
Returning to the original solution, we have shown that
\begin{equation*}
  u\big(t, c_* t + \sqrt{t}\big) \geq \frac{\delta}{t} \e^{-\lambda_* \sqrt{t}} \ForAll t \geq T.
\end{equation*}
That is, we can control $u$ at the diffusive scale.
Furthermore, the comparison principle implies $u(t, \anon)$ is decreasing for all $t \geq 0$.
So
\begin{equation}
  \label{eq:sub-right}
  u\big(t, c_* t + \sqrt{t} - B\big) \geq \frac{\delta}{t} \e^{-\lambda_* \sqrt{t}}
\end{equation}
for all $t \geq T$ and $B \geq 0$.

Before using \eqref{eq:sub-right}, we need a lower bound on $u$ to the left of $x = c_* t$.
This is much simpler.
We extend $f$ by zero to $[-1,1]$.
Then $f$ is a reaction of \emph{ignition type} on this extended interval.
In~\cite{Coville}, Coville shows the existence of a non-increasing front $\ubar U$ solving
\begin{equation*}
  \begin{gathered}
    \mu(J \ast \ubar U - \ubar U\,) + \ubar c \, \ubar U' + f(\ubar U) = 0,\\
    -1 \leq \ubar{U} \leq 1, \quad \ubar U(-\infty) = 1, \quad \ubar U(+\infty) = -1
  \end{gathered}
\end{equation*}
for a unique speed $\ubar c$.
We need a speed strictly less that $c_*$, so let $c' \coloneqq \min\big\{\ubar{c}, c_* - 1\big\}$.
Then $c' \leq \ubar{c}$ implies that $\ubar{U}(x - c't)$ is a subsolution to \eqref{eq:main}.
Hence if we shift $\ubar U$ so that $u(0, \anon) \geq \ubar U$, the comparison principle implies
\begin{equation*}
  u(t,x) \geq \ubar U(x - c' t) \ForAll (t, x) \in [0, \infty) \times \R.
\end{equation*}
It follows that there exists $B \geq 0$ such that
\begin{equation}
  \label{eq:u-left}
  u(t, c' t - B) \geq \frac 1 2 \ForAll t \geq 0
\end{equation}
and
\begin{equation}
  \label{eq:u-init}
  u(T, x) \geq \frac 1 2 \ForAll x \in \big[c' T - B - 1, c_* T + \sqrt{T} - B\big].
\end{equation}

We leverage these bounds to construct a traveling wave subsolution to \eqref{eq:main}.
Crucially, we need a wave with speed $c_*$.
Let $\ti f$ be a function on $\left[0,\frac 1 2\right]$ satisfying $\ti f \leq f$, $\ti f'(0) = f'(0)$, and \ref{hyp:reg}--\ref{hyp:KPP} with $\left[0,\frac 1 2\right]$ in the place of $[0, 1]$.
Then $\ti f$ is a KPP reaction on a restricted interval.
Applying Theorem~1.3 in~\cite{CDM} to the supersolution $\e^{-\lambda_* x}$, there exists a monotone front $\ti U$ satisfying
\begin{equation*}
  \begin{gathered}
    \mu(J\ast \ti U - \ti U) + c_*\ti U' + \ti f(\ti U) = 0,\\
    0 \leq \ti U \leq \frac 1 2, \quad \ti U(-\infty) = \frac 1 2, \quad \ti U(+\infty) = 0.
  \end{gathered}
\end{equation*}
Moreover, Theorem~1.6 in~\cite{CDM} shows
\begin{equation}
  \label{eq:front-decay}
  \ti U(s) \lesssim s \, \e^{-\lambda_*s} \quad \textrm{as } s \to \infty.
\end{equation}
We note that the stated theorem assumes that $J$ is differentiable and $f$ satisfies an additional technical condition.
However, an examination of the proof shows that \eqref{eq:front-decay} holds without these hypotheses when $\ti{U}$ is monotone.

Hence, we can translate $\ti U$ so that
\begin{equation}
  \label{eq:wave-small}
  \ti U\Big(\sqrt{t} + \frac{3}{2\lambda_*}\log(t+1) - B - 1\Big) \leq \frac{\delta}{t} \e^{-\lambda_* \sqrt{t}}
\end{equation}
for all $t \geq T.$
We define
\begin{equation*}
  \ubar w(t,x) \coloneqq \ti U\Big(x - c_*t + \frac{3}{2\lambda_*}\log(t+1)\Big).
\end{equation*}
Let
\begin{equation*}
  \m D \coloneqq \big\{(t,x)\in \R_+ \times \R \mid t > T, \enspace c' t < x + B < c_* t + \sqrt t - 1\big\},
\end{equation*}
and define the augmented boundary $\partial^*\m D$ by
\begin{equation*}
  \partial^* \m D \coloneqq \big\{(t,x) \in \m D^c \mid (t,x') \in \bar{\m D} \textrm{ for some } x'\in \R \textrm{ with } \abs{x - x'} < 1\big\}.
\end{equation*}
Again, this is the subset of $\R_+\times \R$ that can directly influence $\m D$ through the nonlocal kernel $J$.
By \eqref{eq:sub-right}--\eqref{eq:wave-small}, we have constructed $\ubar w$ so that $\ubar w \leq u$ on $\partial^* \m D$.
Furthermore, $\ubar w$ is a subsolution to \eqref{eq:main} because $\ti f \leq f$ and $\ti U$ is decreasing in $x$.
It follows from the comparison principle that
\begin{equation*}
  \ubar w \leq u \textrm{ on } \m D.
\end{equation*}

Now take $\theta \in \left(0,\frac 1 2\right)$.
If we let
\begin{equation*}
  \ubar\sigma_\theta(t) \coloneqq c_* t - \frac{3}{2\lambda_*}\log(t + 1) + \ti U^{-1}(\theta),
\end{equation*}
we have $\ubar w(t, \ubar \sigma_\theta(t)) = \theta$.
Furthermore, $c' < c_*$, so $(t, \ubar \sigma_\theta(t)) \in \m D$ once $t$ is sufficiently large.
Since $u \geq \ubar w$ in $\m D$,
\begin{equation*}
  \sigma_\theta(t) \geq \ubar \sigma_\theta(t)
\end{equation*}
for $t$ sufficiently large.
This is a lower bound for Theorem~\ref{thm:reg}.
For $\theta\in \left[\frac 1 2, 1\right)$, we can repeat the above argument using a different $\ti f$ defined instead on $\left[0, \frac{\theta + 1}{2}\right]$.
This completes the proof of Theorem~\ref{thm:reg}.

\section{Irregular equations}
\label{sec:irreg}

We now turn to Theorem~\ref{thm:trapping} and Proposition~\ref{prop:crit}.
In both cases ${c_* = 0}$, and we are interested in the behavior of stationary fronts.
It is important to note that the continuity and uniqueness of stationary fronts is a delicate issue~\cite{CDM}.
As we shall see, both these pleasant properties can fail in the irregular setting.

\begin{proof}[Proof of Theorem~\textup{\ref{thm:trapping}}.]
  Let $(\mu, J, f)$ be trapping.
  As shown in the introduction, this implies $\op{supp}J \subset [-1,0]$ and $f'(0) > \mu$.
  It follows that the infimum in \eqref{eq:FG} occurs at $\lambda = +\infty$, so $c_* = 0$ and a front $U$ satisfies
  \begin{equation*}
    \mu(J \ast U - U) + f(U) = 0.
  \end{equation*}
  Such fronts need not be unique (even up to translation), but they do exist.
  Indeed, $\op{supp} J \subset [-1,0]$ implies that $\tbf{1}_{\R_-}$ is a supersolution to \eqref{eq:main}.
  By Theorem~1.3 of~\cite{CDM}, a monotone front $U$ exists.

  Define
  \begin{equation*}
    \m F \coloneqq \{s \in [0,1] \mid f(s) > \mu s\}.
  \end{equation*}
  Since $f'(0) > \mu$, $\m F$ contains a nontrivial interval of the form $(0, \theta_0)$.
  But
  \begin{equation*}
    \mu U - f(U) = \mu \, J \ast U \geq 0,
  \end{equation*}
  so $U$ cannot assume any value in $\m F$.
  Since $U(+\infty) = 0$, the profile $U$ must jump \emph{discontinuously} down to $0$ at a finite position.
  By shifting $U$, we may assume that $U(x) = 0$ for all $x \geq 0$.

  Now consider the evolution of $u$ from $\tbf{1}_{\R_-}$.
  We have already noted that $\tbf{1}_{\R_-}$ is a supersolution.
  This is obvious for probabilistic reactions: particles can only jump to the left, so they can never populate $\R_+$.
  On the other hand, the stationary front $U$ is a solution to \eqref{eq:main}, and $U \leq u(0, \anon)$.
  By the comparison principle,
  \begin{equation*}
    U(x) \leq u(t,x) \leq \tbf{1}_{\R_-}(x) \quad \textrm{for all } (t,x) \in [0,\infty) \times \R.
  \end{equation*}
  So $U^{-1}(\theta) \leq \sigma_\theta(t) \leq 0$ for all $\theta \in (0,1)$ and $t \geq 0$.
\end{proof}

\begin{remark}
  More generally, we have shown that $\sigma_\theta(t) = \m O_\theta(1)$ whenever \eqref{eq:main} admits a stationary front that vanishes identically on a positive ray.
  With this observation, we can show that $\sigma_\theta$ is bounded for some critical equations.
  For instance, suppose $\mu = 1,$ $J = \tbf{1}_{[-1,0]}$, and $f(u) = u - u^2$.
  Then a stationary front $U$ satisfies
  \begin{equation}
    \label{eq:unif}
    \int_x^{x+1} U(y) \d y = U(x)^2 \quad \textrm{for all } x \in \R.
  \end{equation}
  Suppose we dictate $U|_{[0, \infty)} \equiv 0$.
  Then for $x\in [-1,0]$, \eqref{eq:unif} becomes
  \begin{equation*}
    \int_x^0 U(y) \d y = U(x)^2.
  \end{equation*}
  Differentiating, we arrive at the ODE $-U = 2UU'$ on $(-1,0)$ with boundary condition $U(0) = 0$.
  By inspection, we have a solution $U(x) = -\frac 1 2 x$ on $[-1,0]$.
  We can then iteratively solve \eqref{eq:unif} on intervals of the form $[-n,1-n]$ for $n\in \N$ to construct a continuous stationary front $U$ that vanishes identically on $\R_+$.
  Thus even in this critical case, $\sigma_\theta(t) = \m O_\theta(1)$.
  Note that this classical reaction is probabilistic: it corresponds to a binary branching random walk.
  Hence this result agrees with~\cite{Bramson-loglog}.

  This argument extends to some non-probabilistic reactions as well.
  Suppose $\mu = 1$, $\op{supp} J \subset [-1, 0]$, and $f(u) \equiv u$ in an open neighborhood of $0$.
  Then $J \ast U = 0$ once $U$ is sufficiently small, so again $U$ vanishes on a positive ray and $\sigma_\theta(t) = \m O_\theta(1)$.
\end{remark}

We now study a special family of critical equations.
\begin{proof}[Proof of Proposition~\textup{\ref{prop:crit}}.]
  By rescaling time, we can reduce to the case $\mu = 1$.
  Then $J = \delta_{-1}$ and $f(u) = u - u^p$ for $p > 1$ fixed.
  A stationary front $U$ satisfies
  \begin{equation*}
    U(x + 1) = U(x)^p \quad \textrm{for all } x \in \R.
  \end{equation*}
  In this case we can explicitly construct a monotone front:
  \begin{equation*}
    U(x) \coloneqq \exp(-p^x).
  \end{equation*}
  In fact, there is an obvious alternative:
  \begin{equation*}
    \ti U(x) \coloneqq U(\floor{x}).
  \end{equation*}
  This equation thus admits nonunique monotone stationary fronts.
  For analytic convenience, we work with $U.$

  As noted in the previous proof, $u = 0$ on $[0,\infty) \times [0,\infty)$.
  Furthermore, we can explicitly compute the solution on $[0, \infty) \times [-1,0)$.
  Indeed,
  \begin{equation*}
    J \ast u(t, x) = u(t, x + 1) = 0
  \end{equation*}
  for all $x \in [-1, 0)$, so
  \begin{equation*}
    \partial_tu = - u^p, \quad u(0, x) = 1.
  \end{equation*}
  Solving this Riccati-type equation, we obtain
  \begin{equation}
    \label{eq:Riccati}
    u(t,x) = \left[(p - 1)t + 1\right]^{-\frac{1}{p - 1}} \quad \textrm{for all } (t,x) \in [0,\infty) \times [-1,0).
  \end{equation}
  In principle, $u$ can be found by iteratively solving ODEs for its values on $[-n, 1 - n)$ with $n \in \N$, but we need not perform such calculations.

  Instead, we construct super- and subsolutions to \eqref{eq:main} on $\R_-$.
  Combining these with the explicit solution \eqref{eq:Riccati} on the ``buffer zone'' $[-1,0)$, we can control $u$ via the comparison principle.
  We begin with the subsolution.
  Define the decreasing shift
  \begin{equation*}
    \sigma_-(t) \coloneqq -U^{-1}(u(t, -1)) + 1
  \end{equation*}
  and
  \begin{equation*}
    \ubar{w}(t,x) \coloneqq U(x - \sigma_-(t)).
  \end{equation*}
  Note that
  \begin{equation*}
    \partial_t \ubar{w} - J \ast \ubar{w} + \ubar{w}^p = \partial_t \ubar{w} = -\dot \sigma_- U' < 0,
  \end{equation*}
  so $\ubar{w}$ is a subsolution to \eqref{eq:main} on $\R_-$.
  By construction,
  \begin{equation*}
    \ubar{w}(0, \anon) = 1 = u(0, \anon) \quad \textrm{on } \R_-
  \end{equation*}
  and
  \begin{equation*}
    \ubar{w}(t, x) \leq \ubar{w}(t, -1) = u(t, x) \quad \textrm{on } [0, \infty) \times [-1, 0).
  \end{equation*}
  By the comparison principle,
  \begin{equation}
    \label{eq:crit-sub}
    u(t,x) \geq \ubar{w}(t,x) \quad \textrm{for all } (t,x) \in [0,\infty) \times \R_-.
  \end{equation}

  We now construct a supersolution.
  Define
  \begin{equation*}
    \sigma_+(t) \coloneqq -\frac{\log\log(t + 1)}{\log p} + 1
  \end{equation*}
  and
  \begin{equation*}
    \bar{w}(t, x) \coloneqq \Omega \, U(x - \sigma_+(t)) = \Omega (t + 1)^{-p^{x - 1}}
  \end{equation*}
  for some $\Omega = \Omega(p) > 1$ to be determined.
  Then
  \begin{equation*}
    \partial_t \bar{w} = - \frac{p^{x - 1}}{t+1} \bar{w}
  \end{equation*}
  and
  \begin{equation*}
    J \ast \bar{w} - \bar{w}^p = - (1 - \Omega^{1 - p}) \bar{w}^p = -(\Omega^{p - 1} - 1) (t + 1)^{-(p - 1) p^{x - 1}} \bar{w}.
  \end{equation*}
  So
  \begin{equation*}
    \partial_t \bar{w} - J \ast \bar{w} + \bar{w}^p = \left[(\Omega^{p - 1} - 1)(t + 1)^{-(p - 1) p^{x - 1}} - p^{x - 1} (t + 1)^{-1}\right] \bar{w}.
  \end{equation*}
  Suppose $x \leq 0$.
  Then
  \begin{equation*}
    (\Omega^{p - 1} - 1)(t + 1)^{-(p - 1) p^{x - 1}} - p^{x - 1} (t + 1)^{-1} \geq \left(\Omega^{p - 1} - 1 - p^{-1}\right) (t + 1)^{-1}.
  \end{equation*}
  We thus choose
  \begin{equation*}
    \Omega = \left(\frac{p + 1}{p}\right)^{\frac{1}{p - 1}},
  \end{equation*}
  so that $\bar{w}$ is a supersolution to \eqref{eq:main}.
  Also,
  \begin{equation*}
    \bar{w}(0,\anon) = \Omega > 1 \geq u(0, \anon).
  \end{equation*}
  When $p \geq 2$,
  \begin{equation*}
    \bar{w}(t, x) \geq \bar{w}(t, 0) = \Omega (t + 1)^{-\frac{1}{p}} \geq [(p - 1) t + 1]^{-\frac{1}{p}} > u(t, x)
  \end{equation*}
  for $x \in [-1, 0)$.
  When $p \in (1, 2)$, the function $(t + 1)^{p - 1}$ is concave in $t$, so
  \begin{equation*}
    (t + 1)^{p-1} \leq (p - 1)t + 1.
  \end{equation*}
  It follows that
  \begin{equation*}
    \bar{w}(t, x) > (t + 1)^{-1} \geq [(p - 1)t + 1]^{-\frac{1}{p - 1}} = u(t, x)
  \end{equation*}
  for $x \in [-1, 0)$.
  Since this holds for all $p > 1$, the comparison principle implies
  \begin{equation}
    \label{eq:crit-super}
    \bar{w}(t,x) \geq u(t,x) \quad \textrm{for all } (t,x) \in [0,\infty) \times \R_-.
  \end{equation}
  By construction, $\ubar{w}$ and $\bar{w}$ are fixed profiles drifting by $\sigma_-$ and $\sigma_+$, respectively.
  Furthermore, there exists $C(p) > 0$ such that
  \begin{equation*}
    \abs{\sigma_\pm(t) + \frac{\log \log t}{\log p}} \leq C \quad \textrm{for all } t \geq 2.
  \end{equation*}
  Thus \eqref{eq:crit-sub} and \eqref{eq:crit-super} imply Proposition~\ref{prop:crit}.
\end{proof}

\section*{Appendix}

In this appendix, we present a proof of Lemma~\ref{lem:key}.
By the Feynman--Kac representation \eqref{eq:FK}, we must control the probability that the log-drifting continuous-time random walk $Y_s^x$ moves from $x$ to $(L, 2L)$ in time $t$ while remaining positive.
The branching random walk literature is littered with such results, but all hold in \emph{discrete} time~\cite{ABR2, BDZ, Mallein}.
Unfortunately, the continuous-time theory is far less developed.
It therefore seems useful to record a proof of Lemma~\ref{lem:key}.

In the course of our proof, we draw on a variety of classical bounds for discrete-time random walks.
We reprove some in our setting, but for others we simply cite the original source, with the understanding that the adaptation to continuous time is transparent.
We choose to follow the ``hands on'' approach of Addario-Berry and Reed in~\cite{ABR2}, which seems quite robust.
Precisely, we modify Theorem~1 in~\cite{ABR2} to handle the continuous-time walk $Y$ with logarithmic drift.

\subsection{Preliminary bounds}

We begin by recalling the definitions of our random walks.
Our central character is $(X_s)_{s \geq 0}$, a continuous-time random walk with jump rate $\nu$, jump law $\bar{K}$ (the spatial reverse of $K$), and constant drift $c_*$ starting from $0$.
In this appendix, let $\m X$ be distributed according to $\bar{K}$.
Then $X_s - c_* s$ is a Poissonization of $\m X$ with rate $\nu s$.
We use this special form to compute the characteristic function $\varphi_{X_s}$ of $X_s$.
\begin{equation}
  \label{eq:characteristic}
  \varphi_{X_s}(\xi) \coloneqq \E \e^{\iu \xi X_s} = \e^{-\nu s + \iu c_* s \xi} \sum_{k = 0}^{\infty} \frac{1}{k!} \left(\nu s \E \e^{\iu \xi \m X}\right)^k = \e^{\nu s [\varphi_{\m X}(\xi) - 1 + \iu m \xi]},
\end{equation}
recalling $m = \frac{c_*}{\nu}$.
We note that the sum in \eqref{eq:characteristic} converges for any $\xi \in \C$, because $\bar{K}$ is compactly supported.
That is, $X_s$ has all exponential moments.
From \eqref{eq:characteristic}, we can compute
\begin{equation}
  \label{eq:moments}
  \E X_s = 0 \And \E X_s^2 = \nu \E \m X^2 s.
\end{equation}
Since $X_s$ has mean zero, the process $(X_s)_{s \geq 0}$ is a martingale.

Given $(t, x) \in [0, \infty) \times \R_+$, we define the log-drifting walk
\begin{equation*}
  Y_s^x \coloneqq X_s + x + D \log \frac{t + 1}{t - s + 1} \quad \textrm{for } 0 \leq s \leq t.
\end{equation*}

To adapt the proof of Theorem~1 in~\cite{ABR2} to $Y$, we require two prior results: Stone's local limit theorem~\cite{Stone} and Lemma~3.3 from Peres and Pemantle~\cite{PP}.
The former is indifferent to logarithmic drifts, so we state it for $X_s$.
It provides upper and lower bounds on the dispersal of $X_s$ on the line.
\begin{proposition}[Stone]
  \label{prop:LLT}
  For each $h \geq 2$, there exists $C_S(h, K, \nu) \geq 1$ such that
  \begin{equation*}
    \P[x \leq X_s \leq x + h] \leq \frac{C_S}{\sqrt{s}}\exp\left[-\frac{x^2}{2 \nu \E \m X^2 \,s}\right] + \smallO_h\big(s^{-\frac 1 2}\big)
  \end{equation*}
  and
  \begin{equation*}
    \P[x \leq X_s \leq x + h] \geq \frac{1}{C_S \sqrt{s}}\exp\left[-\frac{x^2}{2 \nu \E \m X^2 \,s}\right] + \smallO_h\big(s^{-\frac 1 2}\big)
  \end{equation*}
  for all $(s, x) \in \R_+ \times \R.$
  The error terms satisfy $s^{\frac 1 2} \smallO_h\big(s^{-\frac 1 2}\big) \to 0$ as $s \to \infty$ uniformly in $x \in \R$ and $h$ in a compact subset of $[2, \infty)$.
\end{proposition}

\begin{remark}
  This is a simplification of Stone's results, which are more precise when it is known whether $X_s$ is supported in a lattice.
  We remain agnostic on this point, and pay the price with the constant $C_S$.
  The restriction $h \geq 2$ in Proposition~\ref{prop:LLT} is necessary in the lattice case.
\end{remark}

This classical result is typically proved by Fourier analytic methods.
This is convenient in our setting, since $X_s$ has the simple characteristic function \eqref{eq:characteristic}.
Classical proofs thus easily adapt to continuous time, and we direct the reader to~\cite{Stone} for details.

We now turn to the essential bounds of Peres and Pemantle in~\cite{PP}, which concern discrete time random walks.
We require two modification: continuous time and a logarithmic drift.
The latter is significantly more serious.

We begin with bounds for $X_s$.
Define the hitting time
\begin{equation*}
  T_x \coloneqq \inf\{s \geq 0 \mid X_s + x < 0\}
\end{equation*}
parameterized by $x \geq 0$.
Thus $T_x$ is the first time at which $X$ makes an excursion to the left of size $x$.
\begin{lemma}
  \label{lem:hitting}
  There exist $C_1, C_2, c_3 > 0$ depending on $K$ and $\nu$ such that for all $s > 0$\textup{:}
  \begin{enumerate}[label = \textup{(\roman*)}, leftmargin = 5em, labelsep = 1em, itemsep= 1ex, topsep = 1ex]
  \item
    \label{item:hitting-upper}
    $\P[T_x > s] \leq C_1 \max\{x, 1\} s^{-\frac 1 2}$ for all $x \geq 0$;

  \item
    \label{item:hitting-conditioned}
    $\E\big[X_s^2 \mid T_x > s \big] \leq C_2 s$ for all $x \geq 1$;

  \item
    \label{item:hitting-lower}
    $\P[T_x > s] \geq c_3 \min\big\{x s^{-\frac 1 2}, 1\big\}$ for all $x \geq 1$.
  \end{enumerate}
\end{lemma}

\begin{proof}[Proof of Lemma~\ref{lem:hitting}]
  First, we show that these bounds are trivial when $x \geq C \sqrt{s}$ for fixed $C > 0$ to be determined.
  Indeed, \ref{item:hitting-upper} is vacuous in this regime.
  By Kolmogorov's martingale maximal inequality,
  \begin{equation*}
    \P[T_x > s] \geq \P\left[\max_{s' \in [0, \, s]} \abs{X_{s'}} \leq x\right] \geq 1 - \frac{\E X_s^2}{x^2}.
  \end{equation*}
  By \eqref{eq:moments}, $\P[T_x > s] \geq \frac 1 2$ provided $x \geq C \sqrt{s}$ for
  \begin{equation*}
    C^2 = 2 \nu \E \m X^2.
  \end{equation*}
  Since $\E X_s^2 = \frac 1 2 C^2 s$, this proves both \ref{item:hitting-conditioned} and \ref{item:hitting-lower} when $x \geq C \sqrt{s}$.

  We may therefore assume that $x \leq C \sqrt{s}$.
  In discrete time, the ``gambler's ruin'' bounds \ref{item:hitting-upper} and \ref{item:hitting-lower} are well known.
  For instance, they follow from Theorem~5.1.7 in~\cite{LL}.
  That proof easily extends to continuous time, so we do not repeat it.

  This leaves \ref{item:hitting-conditioned} for $1 \leq x \leq C \sqrt{s}$.
  By \ref{item:hitting-upper},
  \begin{equation*}
    \E[T_x \wedge s] = \int_0^s \P[T_x \geq r] \d r \leq 2 C_1 x \sqrt{s}.
  \end{equation*}
  Thus by Wald's identity,
  \begin{equation*}
    \E [X_s^2 \tbf{1}_{T_x > s}] \leq \E X_{T_x \wedge s}^2 = \nu \E \m X^2 \, \E [T_x \wedge n] \leq 2 C_1 \nu \E \m X^2 x \sqrt{s}.
  \end{equation*}
  If we divide by $\P[T_x > s]$, \ref{item:hitting-lower} implies \ref{item:hitting-conditioned}.
\end{proof}

We'd like a version of these bounds for the log-drifting walk $Y$.
However, we need a replacement for Theorem~5.1.7 in~\cite{LL}.
This is the content of Theorem~3.2 in~\cite{PP}.
In continuous time, it reads:
\begin{theorem}
  \label{thm:subdiffusive}
  Let $f \colon [0, \infty) \to [0, \infty)$ be increasing and satisfy
  \begin{equation}
    \label{eq:subdiffusive-criterion}
    \int_0^\infty \frac{f(s)}{(s + 1)^{3/2}} \d s < \infty.
  \end{equation}
  Then there exist $C_4, s_0 \geq 1$ depending on $f$, $K$, and $\nu$ such that for all $s > 0$,
  \begin{equation}
    \label{eq:subdiffusive-lower}
    \P[X_r \geq f(r) \; \textrm{ for all } s_0 \leq r \leq s] \geq \frac{1}{C_4 \sqrt{s}}
  \end{equation}
  and
  \begin{equation}
    \label{eq:subdiffusive-upper}
    \P[X_r \geq -f(r) \; \textrm{ for all } 0 \leq r \leq s] \leq \frac{C_4}{\sqrt{s}}.
  \end{equation}
\end{theorem}

In their proof, Pemantle and Peres use their Lemma~3.3.
Since Lemma~\ref{lem:hitting} is a precise analogue in continuous time, the proof of Theorem~3.2 in~\cite{PP} adapts to continuous time without trouble.
In fact, we prove a generalization of the lower bound \eqref{eq:subdiffusive-lower} below.
We direct the reader to~\cite{PP} for the proof of the upper bound \eqref{eq:subdiffusive-upper}.

Clearly, a logarithmic drift satisfies \eqref{eq:subdiffusive-criterion}, so Theorem~\ref{thm:subdiffusive} constrains $Y_s^x$.
We will also consider $Y$ backwards in time.
We therefore define drifts
\begin{equation*}
  f_1(s) \coloneqq \abs{D} \log \frac{t + 1}{(t - s)_+ + 1} \And f_2(s) \coloneqq \abs{D} \log \frac{1}{s + 1}
\end{equation*}
and hitting times
\begin{equation*}
  S_x^{i \pm} \coloneqq \inf\{s \geq 0 \mid X_s + x \pm f_i(s) < 0\}
\end{equation*}
for $i \in \{1, 2\}$ and $x \geq 0$.
We note that $S_x^{i\pm}$ also depends on $t$ through $f_i$.
Furthermore, $S_x^{i-} \leq T_x \leq S_x^{i+}$ because $f_i \geq 0$.
We use Theorem~\ref{thm:subdiffusive} to establish a version of Lemma~\ref{lem:hitting} for $S$.
\begin{lemma}
  \label{lem:drift-hitting}
  There exist $\ubar x, C_1', C_2', c_3' > 0$ depending on $K$, $\nu$, and $D$ such that for each $i \in \{1, 2\}$ and all $s > 0$\textup{:}
  \begin{enumerate}[label = \textup{(\roman*)}, leftmargin = 5em, labelsep = 1em, itemsep= 1ex, topsep = 1ex]
  \item
    \label{item:drift-hitting-upper}
    $\P[S_x^{i\pm} > s] \leq C_1' \max\{x, 1\} s^{-\frac 1 2}$ \quad for all $x \geq 0$;

  \item
    \label{item:drift-hitting-conditioned}
    $\E \big[X_s^2 \mid S_x^{i\pm} > s \big] \leq C_2' s$ \quad for all $x \geq 1$;

  \item
    \label{item:drift-hitting-lower}
    $\P[S_x^{i\pm} > s] \geq c_3' \min\big\{x s^{-\frac 1 2}, 1\big\}$ \quad for all $x \geq \ubar x$.
  \end{enumerate}
\end{lemma}

\begin{proof}
  \ref{item:drift-hitting-upper}
  It suffices to consider $x \leq \sqrt{s}$.
  Since $\P[S_x^{i-} > s] \leq \P[T_x > s]$, Lemma~\ref{lem:hitting}\ref{item:hitting-upper} allows us to reduce to the cases $S^{i+}$ for $i \in \{1, 2\}$.
  By the central limit theorem, there exists $C \geq 1$ such that
  \begin{equation*}
    \P\big[X_{C (x^2 + 1)} > x\big] \geq \frac 1 3.
  \end{equation*}
  Thus by Lemma~\ref{lem:hitting}\ref{item:hitting-lower} and the FKG inequality (see, for instance,~\cite[\S 2.2]{Grimmett}),
  \begin{equation}
    \label{eq:reaching-diffusive-scale}
    \P\big[X_{C (x^2 + 1)} > x, \, T_1 > C (x^2 + 1)\big] \geq \frac 1 3 \P[T_1 > C (x^2 + 1)] \geq c \min\{x^{-1}, 1\}
  \end{equation}
  for some $c > 0$ which may change from line to line.
  We now form a new drift
  \begin{equation*}
    \ti f(s) \coloneqq
    \begin{cases}
      0 & \textrm{for } s < C (x^2 + 1),\\
      f_i(s - C (x^2 + 1)) & \textrm{for } s \geq C (x^2 + 1).
    \end{cases}
  \end{equation*}
  Let $\ti T_1 \coloneqq \inf\{s \geq 0 \mid X_s + 1 + \ti f(s) < 0\}$ and define the events
  \begin{align*}
    A &\coloneqq \big\{X_{C (x^2 + 1)} > x, \, T_1 > C (x^2 + 1)\big\},\\
    B &\coloneqq \big\{X_{C (x^2 + 1) + r} - X_{C (x^2 + 1)} \geq -x - f_i(r) \; \textrm{ for all } 0 \leq r \leq s\big\}.
  \end{align*}
  Now, disjoint increments of $X$ are independent, and increments of equal length are equidistributed.
  It follows that
  \begin{equation}
    \label{eq:concatenation}
    \P[A] \P\big[S_x^{i+} > s\big] = \P[A, B] \leq \P\big[\ti{T}_1 > C (x^2 + 1) + s\big].
  \end{equation}
  Theorem~\ref{thm:subdiffusive} applies to $\ti f + 1$, so \eqref{eq:subdiffusive-upper} implies
  \begin{equation*}
    \P\big[\ti{T}_1 > C (x^2 + 1) + s\big] \leq \frac{C_4}{\sqrt{C (x^2 + 1) + s}} \leq C_4 s^{-\frac 1 2}.
  \end{equation*}
  Combining this with \eqref{eq:reaching-diffusive-scale} and \eqref{eq:concatenation}, we obtain
  \begin{equation*}
    \P[S_x^{i+} > s] \leq C_1' \max\{x, 1\} s^{-\frac 1 2}.
  \end{equation*}

  \ref{item:drift-hitting-conditioned}
  The continuous-time random walk $X$ obeys the invariance principle; see, for instance,~\cite[Theorem~19.25]{Kallenberg}.
  Therefore, since $f_i(s) = \smallO\big(\sqrt{s}\big)$,
  \begin{equation}
    \label{eq:invariance}
    \P\big[S_x^{i\pm} > s\big] \geq \frac 1 2
  \end{equation}
  provided $x \geq \ubar{C} \sqrt{s}$ for $\ubar{C} > 0$ sufficiently large depending on $K$, $\nu,$ and $D$.
  By \eqref{eq:moment}, we may therefore assume that $1 \leq x \leq \ubar{C} \sqrt{s}$.

  We condition on the time and location of the minimal value of $X$.
  Define the random variable
  \begin{equation*}
    \rho \coloneqq \arginf_{r \in [0, \, s]} X_r.
  \end{equation*}
  Then
  \begin{equation}
    \label{eq:min-position}
    \E\big[X_s^2 \mid S_x^{i \pm} > s\big] \leq \sup_{r, y} \E\big[X_s^2 \mid  S_x^{i \pm} > s, \, \rho = r, \, X_{r-} = y\big],
  \end{equation}
  where the supremum ranges over $r \in [0, s]$ and $y \in [-x \mp f_i(r), 0]$.
  Since $(X_s)_{s \geq 0}$ is c\`adl\`ag and Markov,
  \begin{equation*}
    \E\big[X_s^2 \mid  S_x^{i \pm} > s, \, \rho = r, \, X_{r-} = y\big] = \E\big[(y + X_{s - r})^2 \mid X_u \geq 0 \textrm{ for } u \in [0, s - r]\big].
  \end{equation*}
  But $y^2 \leq C' s$ for some $C' > 0$ fixed, so Young's inequality and Lemma~\ref{lem:hitting}\ref{item:hitting-conditioned} imply
  \begin{equation*}
    \E\big[(y + X_{s - r})^2 \mid X_u \geq 0 \textrm{ for } u \in [0, s - r]\big] \leq C'' s.
  \end{equation*}
  By \eqref{eq:min-position}, we are done.

  \ref{item:drift-hitting-lower}
  The lower bound for $S^{i+}$ follows from Lemma~\ref{lem:hitting}\ref{item:hitting-lower}, so we consider $S^{i-}$.
  By \eqref{eq:invariance}, it suffices to consider $x \leq \ubar{C} \sqrt{s}$.
  We extend the proof of Theorem~3.2(i) in~\cite{PP} to $\ubar x \leq x \leq C \sqrt{s}$.

  For $m \in \N$, we define the event
  \begin{equation*}
    V_m \coloneqq \big\{X_r + x \geq f_i(r)\, \textrm{ for all } r \in (2^{m - 1}, 2^m]\big\}.
  \end{equation*}
  We claim that there exists $\ti C > 0$ such that
  \begin{equation}
    \label{eq:dyadic-success}
    \P\big[V_m^c \mid T_x > 4N\big] \leq \ti C f_i(2^m) 2^{-m/2}
  \end{equation}
  for any $N \geq 2^{m - 1}$.
  To see this, we condition on the value of the first violation:
  \begin{equation*}
    r_* \coloneqq \inf \big\{r \in (2^{m - 1}, 2^m] \mid X_r + x < f_i(r)\big\}.
  \end{equation*}
  Since $f_i$ is increasing,
  \begin{align*}
    \P\big[V_m^c, \, T_x > 4N\big] &\leq \sup_{r \in (2^{m - 1}, 2^m]} \P\big[V_m^c, \, T_x > 4N \mid r_* = r\big]\\
                                  &\leq \P\big[T_x > 2^{m - 1}\big] \sup_r \P\big[X_u - X_r \geq -f_i(2^m) \, \textrm{ for } u \in [r, 4N]\big]\\
                                  &\leq \P\big[T_x > 2^{m - 1}\big] \P\big[T_{f_i(2^m)} > 2N\big].
  \end{align*}
  By Lemma~\ref{lem:hitting}\ref{item:hitting-upper},
  \begin{equation*}
    \P\big[V_m^c, \, T_x > 4N\big] \leq C_1^2 x f_i(2^m) \big(2^m N\big)^{-\frac 1 2}.
  \end{equation*}
  Dividing by $\P\big[T_x > 4N\big]$ and appealing to Lemma~\ref{lem:hitting}\ref{item:hitting-lower}, we obtain \eqref{eq:dyadic-success} with ${\ti C = 2C_1^2c_3^{-1}}$.
  
  Since $f_i$ satisfies \eqref{eq:subdiffusive-criterion} and is increasing, there exists $m_0$ such that
  \begin{equation*}
    \sum_{m = m_0}^\infty f(2^m) 2^{-m/2} \leq \frac{1}{2 \ti C}.
  \end{equation*}
  By \eqref{eq:dyadic-success},
  \begin{equation*}
    \sum_{m = m_0}^M \P\big[V_m^c \;\big|\; T_x > 2^{M + 1}\big] \leq \frac 1 2
  \end{equation*}
  for any $M \geq m_0$.
  Therefore
  \begin{equation*}
    \P\left[\bigcap_{m = m_0}^M V_m \; \bigg| \; T_x > 2^{M + 1} \right] \geq \frac 1 2.
  \end{equation*}
  Multiplying by $\P\big[T_x > 2^{M + 1}\big]$ and using Lemma~\ref{lem:hitting}\ref{item:hitting-lower}, we find
  \begin{equation*}
    \P\left[\bigcap_{m = m_0}^M V_m\right] \geq c x 2^{-\frac{M}{2}}
  \end{equation*}
  for some $c > 0$ which may change from line to line.
  Now, the FKG inequality implies
  \begin{equation*}
    \P\big[S_x^{i-} > 2^M\big] = \P\left[\bigcap_{m = m_0}^M V_m, \; S_x^{i-} > 2^{m_0}\right] \geq c x 2^{-\frac{M}{2}} \P\big[S_x^{i-} > 2^{m_0}\big].
  \end{equation*}
  Finally, we can choose $\ubar{x} \coloneqq \ubar{C} \, 2^{\frac{m_0}{2}}$.
  Then \eqref{eq:invariance} implies ${\P\big[S_x^{i-} > 2^{m_0}\big] \geq \frac 1 2}$.
  Taking $M = \ceil{\log_2 s}$, the proof is complete.
\end{proof}

\begin{corollary}
  \label{cor:conditioned-diffusive}
  There exist $ C_2'', \beta > 0$ and $s_1 \geq 1$ depending on $K$, $\nu,$ and $D$ such that for all ${i \in \{1, 2\}}$, $s \geq s_1,$ and $1 \leq x \leq \sqrt{s}$,
  \begin{equation*}
    \E \Big[(Y_s^x)^2 \mid S_x^{i\pm} > s, \, Y_s^x \geq \beta \sqrt{s}\Big] \leq C_2'' s.
  \end{equation*}
\end{corollary}

\begin{proof}
  Since $x \leq \sqrt{s}$ and $f_i(s) = \smallO\big(\sqrt{s}\big)$, Young's inequality shows that it suffices to bound
  \begin{equation*}
    \E \Big[X_s^2 \mid S_x^{i\pm} > s, \, Y_s^x \geq \beta \sqrt{s}\Big].
  \end{equation*}
  By the central limit theorem, there exist $\beta > 0$ and $s_1 \geq 1$ such that
  \begin{equation*}
    \P\Big[X_s \geq 2\beta \sqrt{s}\Big] \geq \frac 1 3 \ForAll s \geq s_1.
  \end{equation*}
  Since $f_i(s) = \smallO\big(\sqrt{s}\big)$, we can increase $s_1 \geq 1$ to ensure that $f_i(s) \leq \beta \sqrt{s}$ for all $s \geq s_1$.
  Then
  \begin{equation*}
    \P\Big[Y_s^x \geq \beta \sqrt{s}\Big] \geq \P\Big[X_s^x \geq 2\beta\sqrt{s}\Big] \geq \frac 1 3.
  \end{equation*}
  By the FKG inequality,
  \begin{equation*}
    \P\big[Y_s^x \geq \beta \sqrt{s} \mid S_x^{i\pm} > s\big] \geq \frac 1 3.
  \end{equation*}
  Therefore
  \begin{align*}
    \E \Big[X_s^2 \mid S_x^{i\pm} > s, \, Y_s^x \geq \beta \sqrt{s}\Big] &= \frac{\E[X_s^2 \tbf{1}_{\{Y_s^x \geq \beta \sqrt{s}\}} \mid S_x^{i\pm} > s]}{\P[Y_s^x \geq \beta \sqrt{s} \mid S_x^{i \pm} > s]}\\
                                                                 &\leq 3 \E\big[X_s^2 \mid S_x^{i\pm} > s\big] \leq 3 C_2' s
  \end{align*}
  by Lemma~\ref{lem:drift-hitting}\ref{item:drift-hitting-conditioned}.
\end{proof}

\subsection{Proof of Lemma \ref{lem:key}.}

We choose $L = \ubar x$.
Let $I \coloneqq (\ubar x, 2 \ubar x)$.
We define the drift
\begin{equation*}
  f(s) \coloneqq D \log\frac{t + 1}{t - s + 1}
\end{equation*}
as well as the random walk and drifts viewed in reverse:
\begin{equation*}
  \bar{X}_s \coloneqq X_{t - s} - X_t \And \bar f(s) \coloneqq f(t - s) - f(t) = D \log \frac{1}{s + 1}.
\end{equation*}
Although all the lemmas above were stated for $X_s$, they of course apply to $\bar{X}_s$ as well.
Note that $f, \bar f \in \big\{\pm f_i\big\}_{i \in \{1, 2\}}$, with their precise identities depending on the sign of $D$.
Define
\begin{equation*}
  Y_s^x \coloneqq X_s + x + f(s) \And \bar{Y}_s^y \coloneqq \bar{X}_s + y + \bar{f}(s).
\end{equation*}
Finally, define the stopping times
\begin{equation*}
  S_x \coloneqq \inf\big\{s \geq 0 \mid Y_s^x < 0\big\} \And \bar{S}_y \coloneqq \inf\big\{s \geq 0 \mid \bar{Y}_s^y < 0\big\},
\end{equation*}
so that $\big\{S_x, \bar{S}_x\big\} \subset \big\{S_x^{i\pm}\big\}_{i \in \{1, 2\}}$.

With this notation set, we begin with the upper bound.
We wish to control $\P[Y_t^x \in I, S_x > t]$ as a function of $x$ and $t$.
We condition on the final value $y \coloneqq Y_t^x$.
If $Y_t^x \in I$ and $S_x > t$, the following three events must occur:
\begin{equation*}
  S_x \geq \frac{t}{3}, \quad \bar{S}_y \geq \frac{t}{3}, \And y \in I.
\end{equation*}
Using the independence of disjoint increments of $X$, we have
\begin{equation*}
  \P\big[Y_t^x \in I, S_x > t\big] \leq \P\Big[S_x \geq \frac{t}{3}\Big] \, \sup_{y \in I} \P\Big[\bar{S}_y \geq \frac{t}{3}\Big] \, \sup_{z \in \R} \P\big[Y_{2t/3}^x - Y_{t/3}^x + z \in I\big].
\end{equation*}
Here $y$ represents the final position, and $z$ the sum of the increments of $Y^x$ on $[0,t/3]$ and $[2t/3, t]$.
We can handle the first two terms with Lemma~\ref{lem:drift-hitting}\ref{item:drift-hitting-upper}.
For the last, we use a trivial consequence of Proposition~\ref{prop:LLT}: the probability that $X_s$ lands in \emph{any} interval of bounded width is at most $Cs^{-\frac 1 2}$.
Thus
\begin{equation*}
  \P\big[Y_t^x \in I, S_x > t\big] \leq \frac{C \max\{x, 1\}}{\sqrt{t}} \cdot \frac{C}{\sqrt{t}} \cdot \frac{C}{\sqrt{t}} \leq C (x + 1) t^{-\frac 3 2},
\end{equation*}
where we have allowed $C$ to change from expression to expression.

For the lower bound, we employ a similar structure.
We may now assume that $\ubar x \leq x \leq \sqrt{t}$.
Also, it suffices to prove a lower bound for all $t \geq \ubar t$, for some fixed $\ubar{t}$ depending on $K$, $\nu,$ and $D$.
Indeed, by increasing $C_*$, we can ensure that the left side of \eqref{eq:key} is nonpositive when $t < \ubar t$.

We begin with $\ubar{t} = 4\max\{s_0, s_1\}$, though we will increase it steadily over the course of the proof.
Recalling the constants $C_S$ and $\beta$ from Proposition~\ref{prop:LLT} and Corollary~\ref{cor:conditioned-diffusive}, respectively, we define $\al \in [1/4, 1/2)$ by
\begin{equation*}
  1 - 2\al = \min\left\{\frac{\beta^2}{2^9 C_S^2 \,  \nu \E \m X^2}, \, \frac 1 2\right\}.
\end{equation*}
We again condition on the final position $y = Y_t^x$.
If the following events occur, they will ensure that $Y_t^x \in I$ and $S_x > t$:
\begin{itemize}[itemsep = 1ex]
\item $E_1$: $S_x > \al t$ \; and \; $\beta \sqrt{\al t} \leq Y_{\al t}^x \leq \sqrt{C_2'' t}$;

\item $E_2$: $\bar{S}_y > \al t$ \; and \; $\beta \sqrt{\al t} \leq \bar{Y}_{\al t}^0 \leq \sqrt{C_2'' t}$;

\item $E_3$: ${\displaystyle \inf_{\al t \leq s \leq (1 - \al)t} (Y_s^x - Y_{\al t}^x) \geq - Y_{\al t}^x}$;

\item $E_4$: $y \in I$.
\end{itemize}
By the independence of disjoint increments of $X$, we have
\begin{equation*}
  \P\big[Y_t^x \in I,\, S_x > t\big] \geq \P[E_1] \, \inf_{y \in I} \P[E_2] \, \inf_{y \in I} \P[E_3, E_4 \mid E_1, E_2].
\end{equation*}

Lemma~\ref{lem:drift-hitting}\ref{item:drift-hitting-lower} implies $\P[S_x > \al t] \geq \frac{c_3'x}{\sqrt{\al t}}$.
By the choice of $\beta$ in the proof of Corollary~\ref{cor:conditioned-diffusive}, $t \geq \ubar{t} \geq 4s_1$, and FKG, we have
\begin{equation}
  \label{eq:alive-diffusive}
  \P\Big[S_x > \al t,\, Y_{\al t}^x \geq \beta \sqrt{\al t}\Big] \geq \frac{c_3' x}{3\sqrt{\al t}}.
\end{equation}
Also, Chebyshev's inequality and Corollary~\ref{cor:conditioned-diffusive} imply
\begin{equation*}
  \begin{split}
    \P\Big[Y_{\al t}^x > \sqrt{C_2'' t} &\;\; \big | \;\; S_x > \al t, \, Y_{\al t}^x > \beta \sqrt{\al t}\Big]\\
    &\leq \frac{\E\left[(Y_{\al t}^x)^2 \mid S_x > \al t, \, Y_{\al t}^x > \beta \sqrt{\al t}\right]}{C_2'' t} \leq \frac{C_2'' \al t}{C_2'' t} < \frac 1 2.
  \end{split}
\end{equation*}
By \eqref{eq:alive-diffusive}, we obtain
\begin{equation*}
  \P[E_1] \geq c x t^{-\frac 1 2}
\end{equation*}
for some $c$ which may change from line to line.
By identical reasoning,
\begin{equation*}
  \inf_{y \in I} \P[E_2] \geq c t^{-\frac 1 2}.
\end{equation*}
It will thus suffice to show that $\P[E_3, E_4 \mid E_1, E_2] \geq c t^{-\frac 1 2}$.

Let $m \coloneqq (1 - 2\al) t$ denote the length of the middle period.
For $s \in [0, m]$, write
\begin{equation*}
  L_s \coloneqq Y_{\al t + s}^x - Y_{\al t}^x \And R_s \coloneqq Y_{(1 - \al) t - s}^x - Y_{(1 - \al) t}^x.
\end{equation*}
Then $y \coloneqq Y_t^x \in I$ is equivalent to
\begin{equation*}
  L_m \in \bar{Y}_{\al t}^0 - Y_{\al t}^x + I.
\end{equation*}
By the independence of disjoint increments of $X$,
\begin{equation*}
  \P[E_3, E_4 \mid E_1, E_2] \geq \inf_{p, q \in \left[\beta \sqrt{\al t}, \, \sqrt{C_3'' t}\right]} \P\left[L_m \in q - p + I, \, \inf_{s \in [0, \, m]} L_s \geq -p\right].
\end{equation*}
Suppose $p \geq q$ and let $A_{p,q} \coloneqq \{L_m \in q - p + I\}$ and $B_p \coloneqq \big\{\inf_{s \in [0, \, m]} L_s \geq -p\big\}$.
Then
\begin{equation*}
  \P[A_{p,q}, B_p] = \P[A_{p,q}] - \P\big[A_{p, q},  B_p^c\big].
\end{equation*}
We control the first term with Stone's local limit theorem, i.e. Proposition~\ref{prop:LLT}:
\begin{equation}
  \label{eq:Stone-lower-init}
  \P[A_{p,q}] \geq \frac{1}{C_S \sqrt{m}} \exp\left\{-\frac{[q + \ubar{x} - p - f((1 - \al)t) + f(\al t)]^2}{2 \nu \E \m X^2 m}\right\} + \smallO\big(m^{-1/2}\big).
\end{equation}
Recall that $f(t) = \smallO\big(\sqrt{t}\big)$.
We can thus increase $\ubar{t}$ and assume that
\begin{equation}
  \label{eq:F-small}
  F \coloneqq \abs{f(t)} + \ubar{x} \leq \frac{\beta \sqrt{\al t}}{2} \leq \frac{1}{2} \min\{p, q\}.
\end{equation}
Now $F \lesssim \sqrt{t}$ while $p, q \asymp \sqrt{t}$, and $m \asymp t$.
Hence the term in the exponential in \eqref{eq:Stone-lower-init} is bounded.
It follows that we can absorb the $\smallO\big(m^{-1/2}\big)$ error into the main term, provided $\ubar{t}$ is sufficiently large.
Then
\begin{equation}
  \label{eq:Stone-lower}
  \P[A_{p,q}] \geq \frac{1}{2 C_S \sqrt{m}} \exp\left[-\frac{(p - q + F)^2}{2 \nu \E \m X^2 m}\right].
\end{equation}

We will argue that $A_{p, q} \cap B_p^c$ is significantly more unlikely, for then $L$ is forced to make a large excursion.
We can write $B_p^c = U_p \cup V_p \cup W_q$ for
\begin{equation*}
  U_p \coloneqq \left\{\inf_{s \in [0, \, t^{2/3}]} L_s < -p\right\}, \quad V_p \coloneqq \left\{\inf_{s \in [t^{2/3}, \, m/2]} L_s < -p\right\},
\end{equation*}
and
\begin{equation*}
   W_q \coloneqq \left\{\inf_{s \in [m/2, \, m]} R_s < -q\right\}.
\end{equation*}

Now, $U_p$ implies that an increment of $X$ of length at most $t^{2/3}$ drops by at least $p - F$.
By Kolmogorov's maximal inequality and $\al \geq 2^{-2}$,
\begin{equation}
  \label{eq:U-KMM}
  \P[U_p] \leq \frac{\E X_{t^{2/3}}^2}{(p - F)^2} \leq \frac{2^2 \nu \E \m X^2 t^{2/3}}{\beta^2 \al t} \leq \frac{2^4 \nu \E \m X^2}{\beta^2} t^{-\frac 1 3}.
\end{equation}
For $V_p$, we use Proposition~\ref{prop:LLT}:
\begin{equation*}
  \P[V_p] \leq \sup_{r \in [t^{2/3}, \, m/2]} \left[C_S r^{-\frac 1 2} + \smallO\big(r^{-1/2}\big)\right].
\end{equation*}
Now $r \geq t^{2/3}$, so by increasing $\ubar{t}$, we can absorb the error into the main term:
\begin{equation}
  \label{eq:V-Stone}
  \P[V_p] \leq 2 C_S t^{-\frac 1 3}.
\end{equation}
Combining \eqref{eq:U-KMM} and \eqref{eq:V-Stone}, we have
\begin{equation}
  \label{eq:early-drop}
  \P[U_p \cup V_p] \leq C_{1} t^{-\frac 1 3}
\end{equation}
for some constant $C_{1}$ depending only on $K,$ $\nu$, and $D$.
On the other hand, if $A_{p, q}$ still occurs, our walk must climb back up at least to position $q$.
Using the independence of disjoint increments of $X$ and Proposition~\ref{prop:LLT},
\begin{equation*}
  \P[A_{p, q} \mid U_p \cup V_p] \leq \sup_{r \in [0, \, m/2]} \frac{C_S}{\sqrt{m - r}} \exp\left[-\frac{(q - F)^2}{2\nu \E \m X^2 (m - r)}\right] + \smallO\big((m-r)^{-1/2}\big)
\end{equation*}
We wish to compare this to \eqref{eq:Stone-lower}, but the exponent involves $q$ rather than $p - q$.
On the other hand, $p^2, q^2,$ and $m$ are all of the same order.
We have assumed $q \leq p$ and  \eqref{eq:F-small} implies $2F \leq p$.
So
\begin{equation*}
  \frac{(p + F)^2 - (q - F)^2}{2\nu \E \m X^2 m} \leq  \frac{(p + q)(p - q + 2F)}{2\nu \E \m X^2 m} \leq \frac{2 C''}{\nu \E \m X^2 (1 - 2 \al)} \eqqcolon \log C_{2}.
\end{equation*}
It follows that
\begin{equation*}
  \P[A_{p, q} \mid U_p \cup V_p] \leq \frac{2C_S C_{2}}{\sqrt{m}} \exp\left[-\frac{(p - F)^2}{2\nu \E \m X^2 m}\right] + \smallO\big(m^{-1/2}\big).
\end{equation*}
Again, we can absorb the error if we increase $\ubar{t}$.
Using \eqref{eq:early-drop}, it follows that
\begin{equation*}
  \P[A_{p, q} \cap (U_p \cup V_p)] \leq \frac{2^2 C_S C_{1} C_{2}}{t^{1/3} \sqrt{m}} \exp\left[-\frac{(p - q + F)^2}{2\nu \E \m X^2 m}\right].
\end{equation*}
Then if we take $\ubar{t} \geq \left(2^5 C_S^2 C_{1} C_{2}\right)^3,$ we obtain
\begin{equation}
  \label{eq:early-drop-connection}
  \P[A_{p, q} \cap (U_p \cup V_p)] \leq \frac{1}{2^3 C_S \sqrt{m}} \exp\left[-\frac{(p - q + F)^2}{2\nu \E \m X^2 m}\right].
\end{equation}

The case $A_{p, q} \cap W_q$ is simpler.
By the Kolmogorov maximal inequality and our choice of $\al,$
\begin{equation}
  \label{eq:late-drop}
  \P[W_q] \leq \frac{\E X_{m/2}^2}{(q - F)^2} \leq \frac{2 \nu \E \m X^2 m}{\beta^2 \al t} \leq \frac{2^3 \nu \E \m X^2 (1 - 2 \al)}{\beta^2} \leq \frac{1}{2^6 C_S^2}.
\end{equation}
On the other hand, the independence of disjoint increments of $X$ and Proposition~\ref{prop:LLT} imply
\begin{equation}
  \label{eq:late-connection-prelim}
  \P[A_{p, q} \mid W_q] \leq \sup_{r \in [m/2, \, m]} \frac{C_S}{\sqrt{r}} \exp\left[-\frac{(p - F)^2}{2\nu \E \m X^2 r}\right] + \smallO\big(r^{-1/2}\big).
\end{equation}
Now, using $F = \smallO\big(\sqrt{t}\big)$, we can increase $\ubar{t}$ to ensure that
\begin{equation*}
  \frac{(p - q + F)^2}{2\nu \E \m X^2 m} - \frac{(p - F)^2}{2\nu \E \m X^2 r} \leq \frac{(p + F)^2 - (p - F)^2}{2\nu \E \m X^2 m} \leq \frac{C \sqrt{t} F}{t} \leq \log 2.
\end{equation*}
Using this bound in \eqref{eq:late-connection-prelim} and absorbing the error as usual, we find
\begin{equation*}
  \P[A_{p, q} \mid W_q] \leq \frac{2^3 C_S}{\sqrt{m}} \exp\left[-\frac{(p - q + F)^2}{2\nu \E \m X^2 m}\right].
\end{equation*}
By \eqref{eq:late-drop}, this implies
\begin{equation}
  \label{eq:late-drop-connection}
  \P[A_{p, q} \cap W_q] \leq \frac{1}{2^3 C_S \sqrt{m}} \exp\left[-\frac{(p - q + F)^2}{2\nu \E \m X^2 m}\right].
\end{equation}

We can now combine \eqref{eq:early-drop-connection} and \eqref{eq:late-drop-connection}:
\begin{equation*}
  \P\big[A_{p, q}, B_p^c\big] \leq \frac{1}{2^2 C_S \sqrt{m}} \exp\left[-\frac{(p - q + F)^2}{2\nu \E \m X^2 m}\right].
\end{equation*}
By \eqref{eq:Stone-lower}, $\P\big[A_{p, q}, B_p^c\big] \leq \frac{1}{2} \P[A_{p, q}],$ so of course
\begin{equation*}
  \P[A_{p, q}, B_p] \geq \frac 1 2 \P[A_{p, q}] \geq \frac{1}{2^2 C_S \sqrt{m}} \exp\left[-\frac{(p - q + F)^2}{2\nu \E \m X^2 m}\right].
\end{equation*}
So far, we have assumed that $p \geq q$.
The case $q \geq p$ can be handled similarly, so in fact
\begin{equation*}
  \P[A_{p, q}, B_p] \geq \frac{1}{2^2 C_S \sqrt{m}} \exp\left[-\frac{(p + q + F)^2}{2\nu \E \m X^2 m}\right]
\end{equation*}
for any $p, q \in \left[\beta \sqrt{\al t}, \, \sqrt{C_3'' t}\right]$.
As before, the exponent is uniformly bounded and $m \asymp t$, so
\begin{equation*}
  \P[E_3, E_4 \mid E_1, E_2] \geq \inf_{p, q \in \left[\beta \sqrt{\al t}, \, \sqrt{C_3'' t}\right]} \P[A_{p, q}, B_p] \geq \frac{c}{\sqrt{t}}
\end{equation*}
for some $c > 0$ depending on $K,$ $\nu,$ and $D$.
The lower bound in \eqref{eq:key} follows.
\qed

\printbibliography

\end{document}